\documentclass{article}

\usepackage{geometry}

\usepackage{amsmath,amsthm,amsfonts,amssymb,amscd,cite,graphicx,bbm}

\newtheorem{observation}{Observation}
\newtheorem{thm}{Theorem}[section]
\newtheorem{lem}[thm]{Lemma}
\newtheorem{cor}[thm]{Corollary}

\newtheorem{prop}[thm]{Proposition}

\newcommand{\ssat}{\operatorname{ssat}}
\newcommand{\sat}{\operatorname{sat}}
\newcommand{\ex}{\operatorname{ex}}

\title{Extremal bounds for pattern avoidance in multidimensional 0-1 matrices}
\author{Jesse Geneson and Shen-Fu Tsai}

\begin{document}
\maketitle

\begin{abstract}
A 0-1 matrix $M$ contains another 0-1 matrix $P$ if some submatrix of $M$ can be turned into $P$ by changing any number of $1$-entries to $0$-entries. The 0-1 matrix $M$ is $\mathcal{P}$-saturated where $\mathcal{P}$ is a family of 0-1 matrices if $M$ avoids every element of $\mathcal{P}$ and changing any $0$-entry of $M$ to a $1$-entry introduces a copy of some element of $\mathcal{P}$. The extremal function $\ex(n,\mathcal{P})$ and saturation function $\sat(n,\mathcal{P})$ are the maximum and minimum possible number of $1$-entries in an $n\times n$ $\mathcal{P}$-saturated 0-1 matrix, respectively, and the semisaturation function $\ssat(n,P)$ is the minimum possible number of $1$-entries in an $n\times n$ $\mathcal{P}$-semisaturated 0-1 matrix $M$, i.e., changing any $0$-entry in $M$ to a $1$-entry introduces a new copy of some element of $\mathcal{P}$.

We study these functions of multidimensional 0-1 matrices. In particular, we give upper bounds on parameters of minimally non-$O(n^{d-1})$ $d$-dimensional 0-1 matrices, generalized from minimally nonlinear 0-1 matrices in two dimensions, and we show the existence of infinitely many minimally non-$O(n^{d-1})$ $d$-dimensional 0-1 matrices with all dimensions of length greater than $1$. For any positive integers $k,d$ and integer $r\in[0,d-1]$, we construct a family of $d$-dimensional 0-1 matrices with both extremal function and saturation function exactly $kn^r$ for sufficiently large $n$. We show that no family of $d$-dimensional 0-1 matrices has saturation function strictly between $O(1)$ and $\Theta(n)$ and we construct a family of $d$-dimensional 0-1 matrices with bounded saturation function and extremal function $\Omega(n^{d-\epsilon})$ for any $\epsilon>0$. Up to a constant multiplicative factor, we fully settle the problem of characterizing the semisaturation function of families of $d$-dimensional 0-1 matrices, which we prove to always be $\Theta(n^r)$ for some integer $r\in[0,d-1]$.
\end{abstract}




\section{Introduction}
A matrix is called a 0-1 matrix if all of its entries are either $0$ or $1$. The study of extremal theory of 0-1 matrices was motivated in part by the investigation of computational and geometric problems. To find the shortest $L_1$ path between two points in a rectilinear grid with obstacles, Mitchell proposed an algorithm \cite{mitchell1992} with time complexity bounded by the extremal function of a certain forbidden 0-1 matrix which was first obtained by Bienstock and Gy\H{o}ri \cite{BG1991}. Another application was on the maximum number of unit distances among the vertices of a convex $n$-gon raised by Erd\H{o}s and Moser \cite{EM1959}. F\"uredi gave the first upper bound $O(n\log_2 n)$ which is tighter than $n^{1+\epsilon}$ \cite{furedi1990maximum} using the extremal functions of a certain family of 0-1 matrices. Later, Pach and Sharir applied extremal functions of 0-1 matrices to bound the number of pairs of non-intersecting and vertically visible line segments \cite{pach1991vertical}. Perhaps the most well-known application is the resolution of the Stanley-Wilf conjecture in enumerative combinatorics concerning the number of permutations of $[n]$ avoiding another fixed permutation of $[k]$ \cite{klazar2000,MT2003}. Klazar \cite{klazar2000} reduced the problem to showing that forbidden permutation matrices have linear extremal functions, a problem that had been posed by F\"uredi and Hajnal \cite{FH1992}. Marcus and Tardos solved F\"uredi and Hajnal's problem in \cite{MT2003}, thus affirming the Stanley-Wilf conjecture. 

A 0-1 matrix $A$ \textit{contains} another 0-1 matrix $P$ if $A$ has a submatrix that can be transformed to $P$ by flipping any number of $1$-entries to $0$-entries. Otherwise, $A$ \textit{avoids} $P$ and is $P$-\textit{free}. The extremal function $\ex(n, P)$ is the maximum possible number of $1$-entries in a $P$-free $n \times n$ 0-1 matrix. Generalizing to a family of 0-1 matrices $\mathcal{P}$, $\ex(n,\mathcal{P})$ is the maximum possible number of $1$-entries in an $n\times n$ 0-1 matrix that avoids every element of $\mathcal{P}$. F\"uredi and Hajnal \cite{FH1992} and Tardos \cite{Tardos2005} initiated a campaign to determine the asymptotic behavior of the extremal function $\ex(n,P)$ for every forbidden 0-1 matrix with less than five $1$-entries. It is easy to see that $\ex(n, P) \ge n$ for any forbidden 0-1 matrix $P$ with at least two entries and a positive number of $1$-entries. However, there are forbidden 0-1 matrices $P, Q, R$ with four $1$-entries which have $\ex(n, P) = n \alpha(n)$, $\ex(n, Q) = n \log n$, and $\ex(n, R) = n^{3/2}$, where $\alpha(n)$ denotes the inverse Ackermann function \cite{Tardos2005}. 

A 0-1 matrix $A$ is $P$-\textit{saturated} for a 0-1 matrix $P$ if $A$ is $P$-free and changing any $0$-entry of $A$ produces a new 0-1 matrix that contains $P$. Matrix $A$ is $P$-\textit{semisaturated} if changing any $0$-entry of $A$ introduces a new copy of $P$ in $A$. Generalizing to a family of 0-1 matrices $\mathcal{P}$, a 0-1 matrix $A$ is $\mathcal{P}$-saturated if $A$ avoids every element of $\mathcal{P}$ and changing any $0$-entry of $A$ produces a 0-1 matrix that contains some element of $\mathcal{P}$. Matrix $A$ is $\mathcal{P}$-semisaturated if changing any $0$-entry of $A$ introduces a new copy of some element of $\mathcal{P}$. The saturation function $\sat(n,P)$ is the minimum possible number of $1$-entries in an $n\times n$ $P$-saturated 0-1 matrix, and $\ssat(n,P)$ is the minimum possible number of $1$-entries in an $n\times n$ $P$-semisaturated 0-1 matrix. Both functions can be naturally generalized to any family $\mathcal{P}$ of 0-1 matrices. Brualdi and Cao started the study of the saturation function of two-dimensional forbidden 0-1 matrices \cite{BC2021}. Fulek and Keszegh established that the saturation function of every two-dimensional 0-1 matrix is either bounded or $\Theta(n)$.
\subsection{Past results on extremal functions of forbidden 0-1 matrices}

Since the extremal function is at least linear for all forbidden 0-1 matrices except those with all zeroes or a single entry, F\"uredi and Hajnal \cite{FH1992} posed the problem of characterizing the forbidden 0-1 matrices $P$ for which $\ex(n, P) = O(n)$.

One way to approach F\"uredi and Hajnal's problem is to identify what are known as \textit{minimally nonlinear 0-1 matrices}. These are 0-1 matrices $P$ for which $\ex(n, P) = \omega(n)$, but $\ex(n, P') = O(n)$ for every 0-1 matrix $P'$ strictly contained in $P$. In a sense, minimally nonlinear 0-1 matrices delineate the border between linearity and nonlinearity for the forbidden 0-1 matrix extremal function, since any 0-1 matrix $P$ has a linear extremal function if and only if it contains no minimally nonlinear 0-1 matrix. Tardos \cite{Tardos2005} and Keszegh \cite{Keszegh2009} posed the problem of determining whether there are infinitely many minimally nonlinear 0-1 matrices. This was answered in the affirmative by Geneson \cite{Geneson2009} using a non-constructive proof involving a family of forbidden 0-1 matrices defined by Keszegh \cite{Keszegh2009}. A critical element of Geneson's proof was showing that $\ex(n, P) = O(n)$ for every tuple permutation matrix $P$. Later, the papers \cite{Crowdmath2018, GT2020} determined properties of minimally nonlinear 0-1 matrices such as bounds on the maximum number of $1$-entries and maximum number of columns in any minimally nonlinear 0-1 matrix with $k$ rows, as well as the number of minimally nonlinear 0-1 matrices with $k$ rows.

Besides determining bounds on $\ex(n,P)$ for specific 0-1 matrices $P$, another line of research in this area has investigated operations that can be performed on 0-1 matrices that only change their extremal functions by at most a constant factor. Several such operations can be found in \cite{Tardos2005, Keszegh2009, Pettie2011c}. These operations were also used in Geneson's proof of the existence of infinitely many minimally nonlinear 0-1 matrices. Related to these, Pach and Tardos proved a few operations that increase the asymptote of the extremal function by at most a multiplicative factor of $O(\log n)$ \cite{PT2006}.

In addition to two-dimensional 0-1 matrices, there is also an extension of the extremal function $\ex(n,P)$ to multidimensional 0-1 matrices. As in the two-dimensional case, a $d$-dimensional 0-1 matrix $A$ \textit{contains} a $d$-dimensional 0-1 matrix $P$ if $A$ has a submatrix that can be transformed to $P$ by flipping any number of $1$-entries to $0$-entries. Otherwise, $A$ \textit{avoids} $P$ and is $P$-\textit{free}. The extremal function $\ex(n, P, d)$ is the maximum possible number of $1$-entries in a $P$-free $d$-dimensional 0-1 matrix of dimensions $n \times n \times \dots \times n$. Similarly, for a family $\mathcal{P}$ of $d$-dimensional 0-1 matrices, $\ex(n,\mathcal{P},d)$ denotes the maximum possible number of $1$-entries in a $d$-dimensional 0-1 matrix of dimensions $n\times\dots\times n$ that avoids every element of $\mathcal{P}$. Generalizing from the two dimensional case, it is simple to show that $\ex(n, P, d) \ge n^{d-1}$ for any forbidden $d$-dimensional 0-1 matrix $P$ with at least two entries and a positive number of $1$-entries.

In \cite{KM2007}, Klazar and Marcus proved that $\ex(n,P,d) = O(n^{d-1})$ for any $d$-dimensional permutation matrix $P$, generalizing the result of Marcus and Tardos on two-dimensional permutation matrices \cite{MT2003}. Geneson and Tian \cite{GT2017} showed that $\ex(n,P,d) = O(n^{d-1})$ for any $d$-dimensional tuple permutation matrix $P$, extending Geneson's result on two-dimensional tuple permutation matrices \cite{Geneson2009}. They also obtained nontrivial bounds on the extremal functions of \textit{block permutation matrices}, i.e., Kronecker products of all-ones matrices and permutation matrices, extending a result of Hesterberg \cite{Hesterberg2012} for two-dimensional 0-1 matrices. Furthermore, in the same paper Geneson and Tian significantly tightened the bounds on the limit inferior and limit superior of the sequence $\frac{\ex(n,P,d)}{n^{d-1}}$ for permutation matrices, extending a result of Fox \cite{Fox2013} for two-dimensional 0-1 matrices. There has also been research \cite{Geneson2019, Geneson2021b} on operations which can be performed on $d$-dimensional 0-1 matrices $P$ to obtain new $d$-dimensional 0-1 matrices $P'$ for which $\ex(n, P', d)$ can be bounded sharply in terms of $\ex(n, P, d)$.

\subsection{Past results on saturation and semisaturation functions of forbidden 0-1 matrices}

Brualdi and Cao initiated the investigation of the saturation function of forbidden two-dimensional 0-1 matrices \cite{BC2021}, inspired by the corresponding saturation function for forbidden graphs \cite{EHM1964}. The 0-1 matrix $A$ is said to be \textit{$P$-saturated} for the forbidden 0-1 matrix $P$ if $A$ avoids $P$, but changing any $0$-entry to a $1$-entry in $A$ produces a new 0-1 matrix which contains $P$. The saturation function $\sat(n, P)$ is the minimum possible number of $1$-entries in any $n \times n$ 0-1 matrix which is $P$-saturated. Brualdi and Cao showed that $\ex(n, I_k) = \sat(n, I_k)=(k-1)(2n-(k-1))$, where $I_k$ denotes the $k \times k$ identity matrix. In other words, all $n \times n$ 0-1 matrices that are $I_k$-saturated have the same number of $1$-entries.

Fulek and Keszegh obtained a general upper bound on the saturation function $\sat(n,P)$ in terms of the dimensions of $P$, and proved that the saturation function is either bounded or linear \cite{FK2021}. They found a single 0-1 matrix with a bounded saturation function and posed the problem of finding more forbidden 0-1 matrices with bounded saturation functions. Geneson showed that almost all permutation matrices have bounded saturation functions, raised the question of the saturation function of multidimensional permutation matrices, and obtained results for multidimensional $r\times s\times1\times\ldots\times1$ 0-1 matrices \cite{Geneson2021}. Berendsohn fully characterized the permutation matrices with bounded saturation functions \cite{Berendsohn2021}. 

In addition to their results on saturation functions of forbidden 0-1 matrices, Fulek and Keszegh also introduced a notion of semisaturation for forbidden 0-1 matrices \cite{FK2021}. We say that the 0-1 matrix $A$ is \textit{$P$-semisaturated} for the forbidden 0-1 matrix $P$ if changing any $0$-entry to an $1$-entry in $A$ produces a new copy of $P$. In the definition of semisaturation, note that we do not require that $A$ avoids $P$. The semisaturation function $\ssat(n, P)$ is defined to be the minimum possible number of $1$-entries in any $n \times n$ 0-1 matrix which is $P$-semisaturated. In their paper \cite{FK2021}, Fulek and Keszegh characterized the forbidden 0-1 matrices $P$ with bounded semisaturation functions.

Tsai investigated saturation in $d$-dimensional 0-1 matrices and extended the definition of semisaturation to multidimensional 0-1 matrices \cite{Tsai2023}. He determined the saturation functions of $d$-dimensional identity matrices, generalizing the result of Brualdi and Cao for two-dimensional identity matrices. In the same paper \cite{Tsai2023}, he also generalized Fulek and Keszegh's characterization of two-dimensional 0-1 matrices with bounded semisaturation functions $\ssat(n,P)$ \cite{FK2021} to multidimensional 0-1 matrices.

\subsection{New results on extremal functions}

In the same way that $\ex(n, P) = \Omega(n)$ for all forbidden 0-1 matrices $P$ except those with all zeroes or a single entry, we also have that $\ex(n, P, d) = \Omega(n^{d-1})$ for all forbidden $d$-dimensional 0-1 matrices $P$ except those with all zeroes or a single entry. We investigate an extension of F\"{u}redi and Hajnal's problem, which is characterizing the forbidden $d$-dimensional 0-1 matrices $P$ for which $\ex(n, P, d) = O(n^{d-1})$.

One way to approach this problem is to identify what we call \textit{minimally non-$O(n^{d-1})$ $d$-dimensional 0-1 matrices}. We define a $d$-dimensional 0-1 matrix $P$ to be minimally non-$O(n^{d-1})$ if $\ex(n, P, d) = \omega(n^{d-1})$, but $\ex(n, P', d) = O(n^{d-1})$ for every $d$-dimensional 0-1 matrix $P'$ properly contained in $P$. Note that the minimally non-$O(n^{d-1})$ $d$-dimensional 0-1 matrices delineate the border between $\ex(n, P, d)$ being $O(n^{d-1})$ and $\omega(n^{d-1})$, since any $d$-dimensional 0-1 matrix $P$ satisfies $\ex(n, P, d) = O(n^{d-1})$ if and only if it contains no minimally non-$O(n^{d-1})$ $d$-dimensional 0-1 matrix. 

Note that the property of being minimally non-$O(n)$ is the same as the property of being minimally nonlinear for 0-1 matrices. Moreover, it is easy to see that for any minimally nonlinear 0-1 matrix $P$, we can generate a minimally non-$O(n^{d-1})$ $d$-dimensional 0-1 matrix for any $d > 2$. If $P$ has dimensions $j \times k$, consider the $d$-dimensional 0-1 matrix $Q$ of dimensions $j \times k \times 1 \times \dots \times 1$ which projects to $P$ on the first two dimensions. We have $\ex(n, Q, d) = \omega(n^{d-1})$ since $\ex(n, P) = \omega(n)$, but $\ex(n, Q', d) = O(n^{d-1})$ for every $d$-dimensional 0-1 matrix $Q'$ properly contained in $Q$. Thus $Q$ is minimally non-$O(n^{d-1})$.

In this paper, we identify several minimally non-$O(n^2)$ $3$-dimensional 0-1 matrices. We also obtain an upper bound on the largest dimension of a minimally non-$O(n^{d-1})$ $d$-dimensional 0-1 matrix in terms of its other dimensions, extending a result from \cite{Crowdmath2018}. Furthermore, we bound the number of $1$-entries in a minimally non-$O(n^{d-1})$ $d$-dimensional 0-1 matrix in terms of its dimensions, and we bound the total number of minimally non-$O(n^{d-1})$ $d$-dimensional 0-1 matrices with first $d-1$ dimensions $k_1 \times k_2 \times \dots \times k_{d-1}$. 

We make additional progress on the $d$-dimensional extension of F\"{u}redi and Hajnal's problem by proving that $\ex(n, P, d) = O(n^{d-1})$ for every $d$-dimensional 0-1 matrix with at most three $1$-entries. We also exhibit a new operation that can be performed on $d$-dimensional 0-1 matrices to obtain $(d+1)$-dimensional 0-1 matrices whose extremal functions are on the order of $n$ times the extremal function of the original matrix. We use this operation to show the existence of infinitely many minimally non-$O(n^{d-1})$ $d$-dimensional 0-1 matrices with all dimensions of length greater than $1$.

Finally, we investigate extremal functions of forbidden families of $d$-dimensional 0-1 matrices. We extend Tardos' technique in \cite{Tardos2005} to show that there exists an algorithm to assert that $\ex(n,\mathcal{P},d)=O(1)$ for a given family of $d$-dimensional 0-1 matrices $\mathcal{P}$, i.e., the algorithm terminates when the condition holds. Moreover when the condition does not hold, we have $\ex(n,\mathcal{P},d)\ge n$. The algorithm relies on a family of $d$-dimensional 0-1 matrices $\mathcal{J}_{n,d}\cup \mathcal{D}_{n,d}$ for every positive integer $n$. Moreover, the family is broad enough such that if for some positive integer $n_0$ every element in the family $\mathcal{J}_{n_0,d}\cup \mathcal{D}_{n_0,d}$ contains some element of $\mathcal{P}$, then $\ex(n,\mathcal{P},d)$ needs to be as low as $O(1)$ for a host matrix to avoid every element of $\mathcal{J}_{n_0,d}\cup \mathcal{D}_{n_0,d}$ and every element of $\mathcal{P}$.

We show that the extremal function of a family of $d$-dimensional 0-1 matrices is either $0$ or at least $n^{d-k}$ when the size of the family is less than $d$. Also, for any positive integer $k$ and $d$ and integer $r\in[0,d-1]$ we construct a family of $d$-dimensional 0-1 matrices with extremal function equal to $kn^r$ for sufficiently large $n$. We show that the constructed family has the fewest possible elements. Our new results on extremal functions of forbidden $d$-dimensional 0-1 matrices are contained in Section~\ref{sec:ex}. 

\subsection{New results on saturation and semisaturation functions}
We continue in the direction of \cite{Tsai2023} to show that the semisaturation function of every family of $d$-dimensional 0-1 matrices must be $\Theta(n^k)$ for some integer $k\in[0,d-1]$, and we give a complete characterization of such families. With that, for any positive integer $d$ and every integer $k\in[0,d-1]$ we construct a 0-1 matrix whose semisaturation function is $\Theta(n^k)$. Up to a constant multiplicative factor, these results settle the problem of characterizing the semisaturation functions of families of $d$-dimensional 0-1 matrices.

As for saturation functions, we generalize the method of Fulek and Keszegh in \cite{FK2021} to show that no family of $d$-dimensional 0-1 matrices has saturation function strictly between $O(1)$ and $\Theta(n)$. We propose an interesting construction that gives a family of $d$-dimensional 0-1 matrices with saturation function as small as $O(1)$ and extremal function as large as $\Omega(n^{d-\epsilon})$ for any $\epsilon>0$. Moreover, we show for any positive integer $k$ and integer $r \in [0, d-1]$ that it is possible to construct a family of $d$-dimensional 0-1 matrices with saturation function equal to $k n^r$ for all $n$ sufficiently large.

Our new results on saturation functions and semisaturation functions of forbidden $d$-dimensional 0-1 matrices are contained in Section~\ref{sec:sat}. In Section~\ref{sec:con}, we discuss future directions for research on extremal functions, saturation functions, and semisaturation functions of forbidden patterns in multidimensional 0-1 matrices.

\section{Notation}
The \textit{weight} of a 0-1 matrix $A$ is the number of $1$-entries in $A$, denoted $w(A)$.
For positive integer $d$, denote $\{1,2,\ldots,d\}$ by $[d]$. We denote a $d$-dimensional $n_1\times n_2\times\ldots\times n_d$ matrix by $A=A(x_1,\ldots,x_d)$, where $x_i\in[n_i]$ for each $i\in [d]$. When we \textit{exchange} distinct dimensions $i,j$ of matrix $A$ to obtain matrix $B$, we exchange the side lengths of these two dimensions and make $B(x'_1,x'_2,\dots,x'_d)=A(x_1,x_2,\dots,x_d)$ where  $(x'_1,x'_2,\dots,x'_d)$ is obtained from $(x_1,x_2,\dots,x_d)$ by exchanging the $i^{\text{th}}$ and $j^{\text{th}}$ coordinates. To \textit{replicate} dimension $i$ of a $d$-dimensional matrix $P$ of dimensions $l_1\times l_2\times\dots\times\ l_{d-1}\times l_d$ is to have a $(d+1)$-dimensional matrix $P'$ of dimensions $l_1\times l_2\times\dots\times l_{d-1}\times l_d\times l_i$ where every $1$-entry $P(x_1,x_2,\dots,x_d)$ is replaced by $1$-entry $P'(x_1,x_2,\dots,x_d,x_i)$ and all other entries of $P'$ are $0$-entries.

A \textit{k-dimensional cross section} $L$ of a $d$-dimensional $n_1\times n_2\times\ldots\times n_d$ matrix $A$ is the set of all entries of $A$ whose coordinates on a set $C_L$ of $d-k$ dimensions are fixed. A cross section $L$ of matrix $A$ is a \textit{face} if for every $i\in C_L$, the value of the $i^{\text{th}}$ coordinate of every entry of $L$ is fixed to some $b_i\in\{1,n_i\}$. A face $f_B$ of an $n_{B_1}\times\dots\times n_{B_d}$ 0-1 matrix $B$ and a face $f_A$ of an $n_{A_1}\times\dots\times n_{A_d}$ 0-1 matrix $A$ are \textit{counterparts} of each other if $A$ and $B$ are both $d$-dimensional for some positive integer $d$, $f_A$ and $f_B$ have the same maximal set of dimensions $C_{f_A}=C_{f_B}$ on which their entries have fixed coordinates, and on every dimension $i\in C_{f_A}$ the fixed coordinate $b_{A_i}$ in $f_A$ and the fixed coordinate $b_{B_i}$ in $f_B$ satisfy $b_{A_i}=1$ and $b_{B_i}=1$ or $b_{A_i}=n_{A_i}$ and $b_{B_i}=n_{B_i}$.
An $i$-\textit{row} of matrix $A$ is a cross section $L$ with $C_L=[d]\setminus\{i\}$.
An $i$-\textit{layer} of matrix $A$ is a cross section $L$ with $C_L=\{i\}$. Two cross sections $L$ and $K$ are \textit{orthogonal} to each other if $C_L\not\subseteq C_K$ and $C_K\not\subseteq C_L$.
Cross section $g$ is $k$-\textit{orthogonal} to cross section $f$ if $g$ and $f$ are orthogonal to each other, and $|C_g\setminus C_f|\geq k$. Note that if cross sections $g$ and $f$ are orthogonal to each other, then $g$ is $1$-orthogonal to $f$ and $f$ is $1$-orthogonal to $g$.

We represent a $3$-dimensional matrix by specifying all its $1$-layers sorted by their first coordinates. For example the following $3$-dimensional matrix $A$ is of dimensions $2\times3\times1$.
$$
\begin{pmatrix}
    \begin{pmatrix}
      A(1,1,1) \\
      A(1,2,1) \\
      A(1,3,1) \\
    \end{pmatrix};
    \begin{pmatrix}
      A(2,1,1) \\
      A(2,2,1) \\
      A(2,3,1) \\
    \end{pmatrix}
\end{pmatrix}
$$

\section{Extremal functions of $d$-dimensional 0-1 matrices}\label{sec:ex}

We begin with the fundamental \textit{monotonicity} property of extremal function with respect to pattern containment. Although it is stated in terms of 0-1 matrices, it applies to any kind of pattern such as graph, ordered graph, poset, etc. as long as containment is transitive, i.e., if pattern $A$ contains pattern $B$ and pattern $B$ contains pattern $C$ then pattern $A$ contains pattern $C$. The proposition is simple and well-known, so it is hard to find its exact reference.

\begin{prop}
    \label{prop:mono}
    If $P$ and $Q$ are $d$-dimensional 0-1 matrices and $P$ contains $Q$, then $\ex(n,P,d)\ge\ex(n,Q,d)$.
\end{prop}
\begin{proof}
    It suffices to show that the set of $P$-free matrices is the superset of the set of $Q$-free matrices. This is true, because if any $d$-dimensional matrix avoids $Q$ then it cannot contain $P$.
\end{proof}

Next we state a basic quantitative fact of the extremal functions of $d$-dimensional 0-1 matrices.
\begin{prop}\cite{GT2017}
\label{prop:2ones}
   If $P$ is a $d$-dimensional 0-1 matrix with at least two $1$-entries, then $n^{d-1}\le\ex(n,P,d)$.
\end{prop}
\begin{proof}
    Without loss of generality suppose that $P$ has two $1$-entries that differ in the first coordinate. Construct a $d$-dimensional 0-1 matrix $M$ of dimensions $n\times n\times\dots\times n$ where $M(x_1,\dots,x_d)=1$ if and only if $x_1=1$. The result follows as $M$ avoids $P$ and has $n^{d-1}$ $1$-entries.
\end{proof}

Define a \textit{projection} $\bar{P}$ of a $d$-dimensional 0-1 matrix $P$ \textit{along} dimension $d$, or \textit{on} the first $d-1$ dimensions, to be the $(d-1)$-dimensional 0-1 matrix obtained from $P$ by setting $\bar{P}(x_1,\ldots,x_{d-1})=1$ if and only if there exists $x_d$ such that $P(x_1,\ldots,x_{d-1},x_d)=1$. We can define a projection along any other dimension in a similar way. For distinct integers $i,j\in[d]$ define a $2$-dimensional projection $P'$ of $P$ on dimensions $i,j$ as the $2$-dimensional 0-1 matrix obtained from $P$ by setting $P'(k,l)=1$ if and only if there exists a $1$-entry $P(x_1,\dots,x_d)$ where $x_i=k$ and $x_j=l$.

In \cite{GT2017}, projection is defined as collapsing all but the first two dimensions, and Lemma 4.6 of \cite{GT2017} gives a lower bound of the extremal function of $P$ in terms of the extremal function of its projection. Successive application of the following lemma also arrives at the same lower bound.
\begin{lem}\label{lem:projection}
$\ex(n,P,d)=\Omega(n \ex(n,\bar{P},d-1))$.
\end{lem}
\begin{proof}
We construct a $d$-dimensional 0-1 matrix $M$ that avoids $P$ as follows. All $d$-layers of $M$ are identical with $\ex(n,\bar{P},d-1)$ $1$-entries and avoid $\bar{P}$. If $M$ contained $P$, then each $d$-layer of $M$ would contain $\bar{P}$,  a contradiction.
\end{proof}

\begin{cor}
  \label{cor:2dprojection}
  Let $P'$ be a $2$-dimensional projection of a $d$-dimensional matrix $P$ where $d>2$. Then $\ex(n,P,d)=\Omega(n^{d-2}\ex(n,P'))$.
\end{cor}

In the next theorem, we exhibit a new operation on $d$-dimensional 0-1 matrices $P$ which produces a $(d+1)$-dimensional 0-1 matrix whose extremal function is on the order of $n$ times the extremal function of $P$. The operation can be described as using the new dimension to produce a diagonal version of the original forbidden pattern.

\begin{thm}\label{thm:stretch}
Suppose that $P$ is a $d$-dimensional 0-1 matrix of dimensions $k_1 \times \dots \times k_d$ and $P’$ is a $(d+1)$-dimensional 0-1 matrix of dimensions $k_1 \times \dots \times k_d \times k_d$ obtained from $P$ by replacing each $1$-entry at coordinate $(x_1, \dots, x_d)$ with a $1$-entry at coordinate $(x_1, \dots, x_d, x_d)$. Then $\ex(n, P’, d+1) = \Theta(n \ex(n, P, d))$. 
\end{thm}

\begin{proof}
The lower bound $\ex(n, P’, d+1) = \Omega(n \ex(n, P, d))$ follows from Lemma~\ref{lem:projection}. For the upper bound, consider a $P’$-free $(d+1)$-dimensional 0-1 matrix $A$ of dimensions $n \times n \times \dots \times n$. We split $A$ into $2n-1$ tilted cross sections $A_{-n+1}, \dots, A_0, \dots, A_{n-1}$ such that $A_i$ consists of the entries $(x_1, \dots, x_d, x_{d+1})$ of $A$ with $x_{d+1} = x_{d}+i$. Note that $A_i$ consists of $(n-|i|)n^{d-1}$ entries for all $i = -n+1, \dots, n-1$.

If $B_i$ denotes the $d$-dimensional 0-1 matrix obtained from $A_i$ by projecting onto the first $d$ dimensions, note that $B_i$ must avoid $P$ or else $A$ would contain $P’$. Furthermore, $B_i$ and $A_i$ have the same number of entries since any two distinct entries $(x_1, \dots, x_d, x_{d+1})$ and $(x’_1, \dots, x’_d, x’_{d+1})$ in $B_i$ must have $(x_1, \dots, x_d) \neq (x’_1, \dots, x’_d)$. Thus $B_i$ has at most $\ex(n, P, d)$ $1$-entries for all $i = -n+1, \dots, n-1$, so $A$ has at most $(2n-1)\ex(n, P, d) = O(n \ex(n, P, d))$ $1$-entries.
\end{proof}

By applying Theorem~\ref{thm:stretch} $k$ times we get the following corollary.
\begin{cor}\label{thm:general-stretch}
Suppose that $P$ is a $d$-dimensional 0-1 matrix of dimensions $l_1 \times \dots \times l_d$ and $P’$ is a $(d+k)$-dimensional 0-1 matrix obtained from $P$ by replicating the last dimension $k$ times, i.e., $P'$ is of dimensions $l_1 \times \dots \times l_d \times l_d \times \dots \times l_d$ such that $P'(x_1,\dots,x_{d+k})=1$ if and only if $P(x_1,\dots,x_d)=1$ and $x_d=x_{d+1}=\dots=x_{d+k}$.
Then $\ex(n, P’, d+k) = \Theta(n^k \ex(n, P, d))$. 
\end{cor}

The corollary above implies that there are infinitely many minimally non-$O(n^{d-1})$ $d$-dimensional 0-1 matrices with all dimensions of length greater than $1$. Note that one can construct infinitely many minimally non-$O(n^{d-1})$ $d$-dimensional 0-1 matrices by embedding all minimally nonlinear matrices in $d$-dimensional matrices, and each of the resulting matrices has only two dimensions of length greater than $1$. In a sense, the construction in the proof of the following corollary produces minimally non-$O(n^{d-1})$ $d$-dimensional 0-1 matrices that are less degenerate.
\begin{cor}
There are infinitely many minimally non-$O(n^{d-1})$ $d$-dimensional 0-1 matrices with all dimensions of length greater than $1$.
\end{cor}
\begin{proof}
The statement holds for $d=2$. For $d>2$, apply the operation in Corollary~\ref{thm:general-stretch} to each minimally nonlinear $2$-dimensional 0-1 matrix $P$ to obtain a $d$-dimensional 0-1 matrix $P'$. By Corollary~\ref{thm:general-stretch} we have $\ex(n,P',d)=\omega(n^{d-1})$. Any proper submatrix of $P'$ is equal to the result of applying the same operation to some proper submatrix of $P$ and thus has extremal function $O(n^{d-1})$. Therefore the construction gives infinitely many minimally non-$O(n^{d-1})$ $d$-dimensional 0-1 matrices with all dimensions of length greater than $1$.
\end{proof}

The following two lemmas from \cite{GHLNPW2019} generalize their two-dimensional counterparts from \cite{FH1992}.

\begin{lem}\label{lem:add-one} \cite{GHLNPW2019} Let a $d$-dimensional 0-1 matrix $P'$ be obtained from a $d$-dimensional 0-1 matrix $P$ by adding a $1$-layer with a single $1$-entry that is adjacent to another $1$-entry. Then $\ex(n,P',d)\leq n^{d-1}+\ex(n,P,d)$.
\end{lem}

\begin{lem}\label{lem:insertbetween} \cite{GHLNPW2019}
Suppose that $P$ is a $d$-dimensional 0-1 matrix with two consecutive $1$-entries in the same $1$-row. If $P’$ is obtained from $P$ by adding $t$ extra $1$-layers of entries to $P$ between two adjacent $1$-layers of entries of $P$, such that each new $1$-layer has a single $1$-entry in the same $1$-row as the other new $1$-layers and the newly introduced $1$-entries have an $1$-entry from $P$ adjacent to them at both ends, then $\ex(n, P, d)) \le \ex(n, P’, d) \le (t + 1) \ex(n, P, d)$.
\end{lem}

We derive a similar but useful lemma as a corollary of Lemma~\ref{lem:add-one}. In particular, below we define an operation of \textit{lowering} a \textit{bottom} entry along the first dimension. One can define a similar operation of \textit{lifting} a \textit{top} entry by moving it to the opposite direction. Moreover, lowering and lifting along any other dimension can also be defined.

\begin{lem}\label{lem:lower}
Let $P$ be a $d$-dimensional $p_1\times p_2\times\ldots p_d$ 0-1 matrix with some entry $P(p_1,x_2,\ldots,x_d)=1$. Obtain $P'$ from $P$ by lowering a bottom $1$-entry of $P$ as follows.  First, set $P'=P$. Second, attach an empty $1$-layer to the end of $P$. Third, set $P'(p_1,x_2,\ldots,x_d)=0$, and $P'(p_1+1,x_2,\ldots,x_d)=1$. Then we have $\ex(n,P',d) \le \ex(n,P,d)+n^{d-1}$. Moreover if $P$ has at least two $1$-entries then $\ex(n,P',d)=O(\ex(n,P,d))$.
\end{lem}

\begin{proof}
Obtain $P''$ from $P'$ by setting $P''(p_1,x_2,\ldots,x_d)=1$. Then $\ex(n,P',d)\leq \ex(n,P'',d)\le \ex(n,P,d)+n^{d-1}$, with the last inequality following from Lemma~\ref{lem:add-one}.

If $P$ has at least two $1$-entries, by Proposition~\ref{prop:2ones} $\ex(n,P,d)\ge n^{d-1}$. Therefore $\ex(n,P',d)=O(\ex(n,P,d))$.
\end{proof}

Successive applications of Lemma~\ref{lem:add-one} and Lemma~\ref{lem:lower} imply the extremal function of the following $3$-dimensional 0-1 matrix is $\Theta(n^2)$.
\begin{cor}
$\ex\left(n, \left(
			\begin{pmatrix}
			0 & 1 & 0\\
			0 & 0 & 0\\
			0 & 1 & 0
			\end{pmatrix};
			\begin{pmatrix}
			0 & 0 & 0\\
			1 & 0 & 1\\
			0 & 0 & 0
			\end{pmatrix}
			\right),3\right)=\Theta(n^2)$
\end{cor}
\begin{proof}
   By Proposition~\ref{prop:2ones}
   the extremal function of the matrix above is at least $n^{d-1}$, so
   it suffices to show that the extremal function of interest is $O(n^2)$.

    We start from a $3$-dimensional matrix
    $$
    P_1=\begin{pmatrix}
    \begin{pmatrix}
    1\\
    1\\
    1
    \end{pmatrix}
    \end{pmatrix}
    $$
    We have $\ex(n,P_1,3)\le 2n^2$ because no $2$-row of any $n\times n\times n$ $P_1$-free matrix could have $3$ or more $1$-entries. 

    Then, we apply Lemma~\ref{lem:lower} to $P_1$ along the first dimension to lower the middle $1$-entry $P_1(1,2,1)$. Hence $\ex(n,P_2,3)=O(n^2)$ where
    $$
    P_2=\begin{pmatrix}
    \begin{pmatrix}
    1\\
    0\\
    1
    \end{pmatrix};
    \begin{pmatrix}
    0\\
    1\\
    0
    \end{pmatrix}
    \end{pmatrix}
    $$

    We further apply Lemma~\ref{lem:add-one} to $P_2$ along the third dimension to add a $1$-entry. Thus $\ex(n,P_3,3)=O(n^2)$ where
    $$
    P_3=\begin{pmatrix}
    \begin{pmatrix}
    0 & 1\\
    0 & 0\\
    0 & 1
    \end{pmatrix};
    \begin{pmatrix}
    0 & 0\\
    1 & 1\\
    0 & 0
    \end{pmatrix}
    \end{pmatrix}
    $$

    Finally we apply Lemma~\ref{lem:lower} to $P_3$ along the third dimension to lower the $1$-entry $P_3(2,2,2)$. This results in the 0-1 matrix in the statement, and therefore the extremal function of interest is $O(n^2)$.
\end{proof}

The next two lemmas generalize a result from \cite{Tardos2005}. A similar generalization was stated in \cite{GHLNPW2019} without proof. 

\begin{lem}\label{lem:attach1}
Suppose that $P$ is a $d$-dimensional 0-1 matrix, and let $P'$ be obtained by attaching an empty $1$-layer to $P$. Then $\ex(n,P',d)=O(\ex(n,P,d)+n^{d-1})$.
\end{lem}
\begin{proof}
Suppose that the last $k$ $1$-layers of $P$ are empty, and the $(k+1)^\text{th}$ to last $1$-layer of $P$ is not empty. For any $n\times n\times\ldots\times n$ 0-1 matrix $M'$ that avoids $P'$, let $M$ be obtained by deleting all $1$-entries in the last $k+1$ $1$-layers from $M'$. Clearly $M$ avoids $P$, and $w(M')\leq w(M)+(k+1)n^{d-1}$.
\end{proof}

\begin{lem}\label{lem:insert1}
Suppose that $P$ is a $d$-dimensional 0-1 matrix, and let $P'$ be obtained by inserting an empty $1$-layer in $P$. Then $\ex(n,P',d)=O\left(\ex(n,P,d)\right)$.
\end{lem}
\begin{proof}
Suppose that $P$ has no more than $k-2$ consecutive empty $1$-layers. For any $n\times n\times\ldots\times n$ 0-1 matrix $M$ that avoids $P'$, for each $p\in\mathbb{Z}_k$ we construct an $n\times n\times\ldots\times n$ 0-1 matrix $M_p$ by retaining every $1$-entry $M(x_1,x_2,\ldots,x_d)$ where $x_1 \equiv p\mod k$, and setting all other entries to zero. Clearly $M_p$ avoids $P$, or else $M$ would contain $P'$, so $w(M)=\sum_{p\in\mathbb{Z}_k}w(M_p)\leq k\times \ex(n,P,d)$.
\end{proof}

In the next two lemmas, we generalize two operations from \cite{PT2006} that cause the extremal function to grow by at most a factor of $O(\log n)$.
\begin{lem}
Let $P$ be a $d$-dimensional 0-1 matrix for which there exist two adjacent $1$-layers $l_1,l_2$ in $P$ where $l_1$ is before $l_2$, $l_1$ has an $1$-entry $o_1$ that is in the same $1$-row as a $1$-entry $o_2$ in $l_2$, and $l_2$ has another $1$-entry $o_3$. Construct a $d$-dimensional 0-1 matrix $P'$ from $P$ by inserting a $1$-layer between $l_1$ and $l_2$ with a single $1$-entry $o$ that is in the same $1$-row as $o_3$. Then we have $\ex(n,P',d)=O(\ex(n,P,d)\log n)$.
\end{lem}
\begin{proof}
Suppose that $M$ is a $d$-dimensional 0-1 matrix of dimensions $n\times\dots\times n$ that avoids $P'$. For each integer $l\in\left[0,\lfloor\log_2n\rfloor\right]$, define a $d$-dimensional 0-1 matrix $M_{l}$ of the same dimensions as $M$ by keeping only every $1$-entry $e$ in $M$ such that there exists another $1$-entry $e_1$ in the same $1$-row as $e$ with a smaller first coordinate, and their first coordinate difference is in $[2^l,2^{l+1})$. Then, obtain from $M_{l}$ a $d$-dimensional 0-1 matrix $N_{l}$ of the same dimensions by deleting the second, fourth, sixth, etc. $1$-entries in every $1$-row of $M_{l}$. Any two consecutive $1$-entries in every $1$-row of $N_{l}$ are at least $2^{l+1}$ positions apart.

Matrix $N_{l}$ avoids $P$. Suppose it does not, then it has a copy of $P$ with $1$-entries $p_1,p_2,p_3$ matching $o_1,o_2,o_3$, respectively. The first coordinates of $p_1$ and $p_2$ are at least $2^{l+1}$ positions apart. In $M$ there exists a $1$-entry $p$ in the same $1$-row as $p_3$ with a smaller first coordinate such that the difference of their first coordinates is in $[2^l,2^{l+1})$. Thus $p$ and $P$ forms a copy of $P'$.

Finally, we have
\begin{align*}
w(M)&\leq n^{d-1}+\sum_{l=0}^{\lfloor\log_2n\rfloor}w(M_{l})\leq n^{d-1}+2\sum_{l=0}^{\lfloor\log_2n\rfloor}w(N_{l})\\
		&\leq n^{d-1}+2(\lfloor\log_2n\rfloor+1)\ex(n,P,d)\\
&=O(\ex(n,P,d)\log n).
\end{align*}
\end{proof}

The next lemma which generalizes a result from \cite{PT2006} can be proved very similarly by forming a series of matrices $M_{l,k}$ from a $P'$-free matrix $M$ according to the distances from each $1$-entry to the closest $1$-entries in the same $1$-row with smaller and larger first coordinates. Hence we skip the detailed proof.

\begin{lem}
Let $P$ be a $d$-dimensional 0-1 matrix for which there exist two adjacent $1$-layers $l_1,l_2$ in $P$ where $l_1$ has $1$-entries $o_1,o_2$ and possibly more $1$-entries, $l_2$ has $1$-entries $o_3,o_4$ and possibly more $1$-entries, and $o_1,o_3$ are in the same $1$-row. Construct a $d$-dimensional 0-1 matrix $P'$ from $P$ by inserting two $1$-layers between $l_1$ and $l_2$ each with a single $1$-entry, the new $1$-entry with the first coordinate closer to that of $o_3$ is in the same $1$-row as $o_3$, and the other new $1$-entry is in the same $1$-row as $o_4$. Then we have $\ex(n,P',d)=O(\ex(n,P,d)\log^2 n)$.
\end{lem}
\begin{lem}\label{lem:threeone}
For any $d$-dimensional 0-1 matrix $P$ with no more than three $1$-entries, $\ex(n,P,d)=O(n^{d-1})$.
\end{lem}
\begin{proof}
  We prove it by induction on $d$. The statement holds for $d=2$. For $d>2$, according to Lemma~\ref{lem:attach1} and Lemma~\ref{lem:insert1} we can assume that $P$ has no empty $1$-layers and thus has at most three $1$-layers. If $P$ has one $1$-layer, let $\bar{P}$ be the projection of $P$ along the first dimension. Every $1$-layer of a $n\times\dots\times n$ $d$-dimensional 0-1 matrix $M$ that avoids $P$ has at most $\ex(n,\bar{P},d-1)$ $1$-entries, since otherwise this $1$-layer would contain $\bar{P}$ and $M$ would contain $P$. So $\ex(n,P,d)\le n\ex(n,\bar{P},d-1)$. Matrix $\bar{P}$ has no more than three $1$-entries, so by induction hypothesis $\ex(n,\bar{P},d-1)=O(d^{n-2})$ and thus $\ex(n,P,d)=O(n^{d-1})$.
  If $P$ has two $1$-layers, denote $\bar{P}$ embedded in $d$-dimensional space by $P'$. By Lemma~\ref{lem:lower} we have $\ex(n,P,d)\leq \ex(n,P',d)+n^{d-1}=O(n^{d-1})$. By symmetry, the same argument applies if $P$ has at most two $i$-layers for any $i$. Otherwise, if $P$ has exactly three $i$-layers for every $i$, then $P$ is a $d$-dimensional permutation matrix, so $\ex(n,P,d)=O(n^{d-1})$ by \cite{KM2007}.
\end{proof}

With the tools above, we give several examples of minimally non-$O(n^2)$ $3$-dimensional 0-1 matrices below. Each of them has weight four, so by Theorem~\ref{lem:threeone} removing any $1$-entry from any of them results in a 0-1 matrix with extremal function $O(n^2)$.

\begin{cor}\label{cor:diag}
For any $3$-dimensional 0-1 matrix $P$ in
$$
\left\{
\left(
			\begin{pmatrix}
			1 & 0 \\
			1 & 0 \\
			\end{pmatrix};
			\begin{pmatrix}
			0 & 1 \\
			0 & 1 \\
			\end{pmatrix}
			\right),
   \left(
			\begin{pmatrix}
			1 & 0 \\
			1 & 0 \\
			\end{pmatrix};
			\begin{pmatrix}
			0 & 1 \\
			0 & 0 \\
			\end{pmatrix};
			\begin{pmatrix}
			0 & 0 \\
			0 & 1 \\
			\end{pmatrix}
			\right)
   \right\},
$$
we have $\ex\left(n,P,3\right)=\Theta(n^{2.5})$ and $P$ is minimally non-$O(n^2)$.
\end{cor}
\begin{proof}
If $P$ is the first matrix, it could be obtained from the $2\times2$ all-ones matrix by replicating the second dimension to the third dimension as described in Theorem~\ref{thm:stretch} followed by exchanging the first two dimensions. By symmetry exchanging dimensions do not affect extremal function, so by Theorem~\ref{thm:stretch} $\ex(n,P,3)=\Theta(n\times n^{1.5})=\Theta(n^{2.5})$ as the $2\times2$ all-ones matrix has extremal function $\Theta(n^{1.5})$ \cite{FH1992}.

If $P$ is the second matrix,
projecting $P$ along the first dimension yields the $2\times2$ all-ones matrix. So by Lemma~\ref{lem:projection} $\ex(n,P,3)=\Omega(n\times n^{1.5})=\Omega(n^{2.5})$. For the upper bound, we start from the first matrix denoted as $Q$. We apply Lemma~\ref{lem:lower} to $Q$ along the first dimension,  lowering entry $Q(2,2,2)$ to obtain $P$.
Thus by Lemma~\ref{lem:lower} $$
\ex(n,P,3)\le\ex(n,Q,3)+n^2=O(\ex(n,Q,3))=O(n^{2.5})
$$.
\end{proof}

\begin{lem}
For any $3$-dimensional 0-1 matrix $P$ in
$$
\left\{
   \left(
			\begin{pmatrix}
			0 & 1 \\
			0 & 0 \\
			\end{pmatrix};
			\begin{pmatrix}
			1 & 0 \\
			1 & 0 \\
			\end{pmatrix};
			\begin{pmatrix}
			0 & 0 \\
			0 & 1 \\
			\end{pmatrix}
			\right),
   \left(
   			\begin{pmatrix}
			0 & 0 \\
			0 & 1 \\
			\end{pmatrix};
			\begin{pmatrix}
			0 & 1 \\
			1 & 0 \\
			\end{pmatrix};
			\begin{pmatrix}
			1 & 0 \\
			0 & 0 \\
			\end{pmatrix}
			\right)
   \right\},
$$
we have $\ex\left(n,P,3\right)=\Theta(n^{2.5})$ and $P$ is minimally non-$O(n^2)$.
\end{lem}
\begin{proof}
For each $P$ above,
projecting $P$ along the first dimension yields the $2\times2$ all-ones matrix, which has extremal function $\Theta(n^{1.5})$ \cite{FH1992}. So by Lemma~\ref{lem:projection} $\ex(n,P,3)=\Omega(n\times n^{1.5})=\Omega(n^{2.5})$.

The first matrix can be obtained from the  $1\times2\times2$ all-ones matrix $R$ by lowering the $1$-entry $R(1,2,2)$ and lifting the $1$-entry $R(1,1,2)$.
The second matrix can be obtained from the  $1\times2\times2$ all-ones matrix $R$ by lowering the $1$-entry $R(1,1,1)$ and lifting the $1$-entry $R(1,2,2)$. In either case by Lemma~\ref{lem:lower} we have $\ex(n,P,3)=O(\ex(n,R,3))$. We conclude by showing that $\ex(n,R,3)=O(n^{2.5})$. If $M$ is a $d$-dimensional matrix of dimensions $n\times n\times\dots\times n$  that avoids $R$, then every $1$-layer of $M$ avoids the $2\times2$ all-ones matrix, which has extremal function $\Theta(n^{1.5})$ \cite{FH1992}. So the weight of $M$ is $O(n\times n^{1.5})=O(n^{2.5})$.
\end{proof}

\begin{lem}
For the $3$-dimensional 0-1 matrix $P=
\left(
			\begin{pmatrix}
			0 & 1 \\
			1 & 0 \\
			\end{pmatrix};
			\begin{pmatrix}
			1 & 0 \\
			0 & 1 \\
			\end{pmatrix}
			\right)$,
$\ex\left(n,P,3 \right)=\Omega(n^{2.5})$, $\ex\left(n,P,3 \right)=O(n^{2.75})$, and $P$ is minimally non-$O(n^2)$.
\end{lem}
\begin{proof}
Projecting $P$ along the first dimension yields the $2\times2$ all-ones matrix, which has extremal function $\Theta(n^{1.5})$ \cite{FH1992}. So by Lemma~\ref{lem:projection} $\ex(n,P,3)=\Omega(n\times n^{1.5})=\Omega(n^{2.5})$.
For the upper bound, $P$ is contained in a $2\times2\times2$ all-ones matrix, which has extremal function $O(n^{2.75})$ \cite{GT2017}. From monotonicity in Proposition~\ref{prop:mono} the upper bound follows.
\end{proof}

Next we prove bounds on the dimensions of minimally non-$O(n^{d-1})$ $d$-dimensional 0-1 matrices. One might hope to obtain a bound between the longest dimension of a minimally non-$O(n^{d-1})$ $d$-dimensional 0-1 matrix and its shortest dimension, but it is impossible to obtain a finite upper bound for $d > 2$. For example, consider any minimally nonlinear 0-1 matrix $P$ of dimensions $j \times k$, and let $R$ be the $d$-dimensional 0-1 matrix of dimensions $j \times k \times 1 \times \dots \times 1$ which has $P$ as a projection on the first two coordinates. Then $R$ is minimally non-$O(n^{d-1})$, but the ratio between its maximum and minimum dimension lengths can be arbitrarily high. Thus we cannot obtain an upper bound only between the longest dimension of a minimally non-$O(n^{d-1})$ $d$-dimensional 0-1 matrix and its shortest dimension, but as we see in the next result, it is possible to obtain an upper bound between the longest dimension and all other dimensions. The next result generalizes a result for minimally nonlinear 0-1 matrices from \cite{Crowdmath2018}.

\begin{thm}\label{thm:minnonlin_kd}
For any minimally non-$O(n^{d-1})$ $d$-dimensional 0-1 matrix of dimensions $k_1 \times k_2 \times \dots \times k_d$ with $k_1 \le k_2 \le \dots \le k_d$, \[k_d \le 1+2 \sum_{i =1}^{d-1} (2k_i - 2).\]
\end{thm}

\begin{proof}
Suppose that $P$ is a minimally non-$O(n^{d-1})$ $d$-dimensional 0-1 matrix with dimensions $k_1 \le k_2 \le \dots \le k_d$. We first partition $P$ into $k_d$ $d$-layers where the entries in $d$-layer $j$ have $d^{\text{th}}$ coordinate $j$ for each $j\in[k_d]$. We can assume that there is no $2$-dimensional projection of $P$ on any two dimensions which is equal to 
$$
\begin{pmatrix} 1 & 0 & 1 & 0 \\
0 & 1 & 0 & 1 \end{pmatrix}
$$
or any of its reflections or rotations. If there is, consider the smallest submatrix $P'$ of $P$ containing the four corresponding $1$-entries. By applying Lemma~\ref{lem:projection} $(d-2)$ times we see that $\ex(n,P',d)=\omega(n^{d-1})$, which contradicts the fact that $P$ is minimally non-$O(n^{d-1})$ unless $P = P'$. If $P = P'$, then clearly we have $k_d \le 1+2 \sum_{i =1}^{d-1} (2k_i - 2)$.

For each $d$-layer of $P$, we cannot have all entries equal to zero since otherwise $P$ would properly contain some non-$O(n^{d-1})$ $d$-dimensional 0-1 matrix by Lemmas~\ref{lem:attach1} and \ref{lem:insert1}. We construct a new $d$-dimensional matrix $Q$ from $P$ with the same dimensions but only a single $1$-entry in each $d$-layer. We scan through the $d$-layers $j$ in the order $j=1$ to $k_d$. For each $d$-layer $j$, we define that $d$-layer $j$ of $Q$ is equal to $d$-layer $j$ of $P$ if $d$-layer $j$ of $P$ only has a single $1$-entry. Otherwise we pick only a single $1$-entry in $d$-layer $j$ of $P$ which differs in some coordinate from the $1$-entry in $d$-layer $j-1$ of $Q$. If there are multiple such $1$-entries, then we pick the $1$-entry with the lexicographically minimal coordinates. We include only the picked $1$-entry in $d$-layer $j$ of $Q$, and we make all other entries in $d$-layer $j$ of $Q$ equal to zero. We use $o_j$ to refer to the $1$-entry in $d$-layer $j$ of $Q$.

Next we construct $d-1$ sequences $S_1, \dots S_{d-1}$ from $Q$. The sequence $S_i$ has first element equal to the $i^{\text{th}}$ coordinate of $o_1$ for each $i\in[d-1]$. For each $j < k_d$, we append the $i^{\text{th}}$ coordinate of $o_{j+1}$ to $S_i$ if it is different from the last element of $S_i$. Observe that if no coordinate of $o_{j+1}$ is appended to any sequence for some $j<k_d$, then the first $d-1$ coordinates of $o_{j}$ and $o_{j+1}$ are identical. Thus for any $j<k_d-1$ it is impossible that both $o_{j+1}$ and $o_{j+2}$ have no coordinate appended to any sequence because that would imply that $o_j,o_{j+1},o_{j+2}$ all have the same $i^{\text{th}}$ coordinate for all $i\in[d-1]$, $P$ does not have any other $1$-entry in $d$-layer $j+1$, and by Lemma~\ref{lem:insertbetween} removing $d$-layer $j+1$ from $P$ would change $\ex(n,P,d)$ by at most a constant factor. Thus $k_d-1$ is at most twice $\sum_{i=1}^{d-1}(|S_i|-1)$ where $|S_i|$ is the length of sequence $S_i$. Each sequence $S_i$ has no immediate repetitive letters and avoids alternating subsequences of length $4$, or else $P$ would contain some $d$-dimensional 0-1 matrix with a $2$-dimensional projection on dimensions $i,d$ equal to
$$
\begin{pmatrix} 1 & 0 & 1 & 0 \\
0 & 1 & 0 & 1 \end{pmatrix}
$$
or one of its reflections or rotations, so $P$ would not be minimally non-$O(n^{d-1})$. Thus sequence $S_i$ has length at most $2k_i-1$ by \cite{DS1965}, so $k_d \le 1 + 2 \sum_{i =1}^{d-1} (2k_i - 2).$
\end{proof}

Given the dimensions of a minimally non-$O(n^{d-1})$ $d$-dimensional 0-1 matrix, we obtain an upper bound on the number of $1$-entries. Like the last theorem, the next result generalizes an upper bound from \cite{Crowdmath2018}.

\begin{thm}
Suppose that $P$ is a minimally non-$O(n^{d-1})$ $d$-dimensional 0-1 matrix of dimensions $k_1 \times k_2 \times \dots \times k_d$ for which $P$ does not have all ones, two dimensions of length $2$, and all other dimensions of length $1$. Then the weight of $P$ is at most \[k_d-1+\prod_{i=1}^{d-1} k_i.\]
\end{thm}

\begin{proof}
Let $P$ be a minimally non-$O(n^{d-1})$ $d$-dimensional 0-1 matrix $P$ such that $P$ does not have all ones, two dimensions of length $2$, and all other dimensions of length $1$. The result is true if $P$
has a $2$-dimensional projection on some two dimensions which is equal to 
$$Q_1=
\begin{pmatrix}
1 & 1 & 0\\
1 & 0 & 1\\
\end{pmatrix}
$$
or any of its reflections or rotations.
This is because $P$ must have exactly four $1$-entries, which does not exceed the upper bound in the statement. If $P$ has more than four $1$-entries, then there exists a $1$-entry $o$ of $P$ and dimensions $i,j$ such that the $2$-dimensional projection of $P$ on dimensions $i,j$ after removal of $o$ is still equal to $Q_1$ or any of its reflections or rotations. Remove $o$ to obtain $Q$. By Corollary~\ref{cor:2dprojection} $\ex(n,Q,d)=\Omega(n^{d-2}\ex(n,Q_1))=\Omega(n^{d-1}\log n)=\omega(n^{d-1})$ as $\ex(n,Q_1)=\Theta(n\log n)$ \cite{furedi1990maximum}. In other words $P$ is not minimally non-$O(n^{d-1})$.
  If $P$ has at least three dimensions of length greater than $1$ and has a $2$-dimensional projection on some two dimensions equal to the $2 \times 2$ all-ones matrix, then by similar arguments and that the extremal function of the $2\times 2$ all-ones matrix is $\Theta(n^{1.5})=\omega(n)$ \cite{FH1992} matrix $P$ must have exactly four $1$-entries and thus the result is also true. Hence, we can assume that there is no $2$-dimensional projection of $P$ on any two dimensions which is equal to the $2 \times 2$ all-ones matrix, $Q_1$, 
or any of its reflections or rotations, since $P$ is minimally non-$O(n^{d-1})$. Given $P$, let $P'$ be the $d$-dimensional 0-1 matrix obtained from $P$ by changing the first $1$-entry in each $d$-row to a $0$-entry. Note that $P'$ cannot have any $d$-layer with multiple $1$-entries, or else there would be some $2$-dimensional projection of $P$ on some two dimensions to be described below which is equal to the $2 \times 2$ all-ones matrix,
$Q_1$,
or one of its reflections or rotations. This is because if $P'$ has two $1$-entries $o_1$ and $o_2$ in the same $d$-layer, then suppose that $o_1$ and $o_2$ have different $i^{\text{th}}$ coordinates. The projection of $o_1$, $o_2$, and the removed $1$-entries in the same $d$-rows as $o_1$ and $o_2$ to dimensions $d$ and $i$ must be the $2\times2$ all-ones matrix, $Q_1$, or one of their reflections or rotations. This contradicts our assumption about $P$. Since $P'$ cannot have any $1$-entry in its first $d$-layer, $P'$ has at most $k_d - 1$ $1$-entries. Thus the number of $1$-entries in $P$ is at most \[k_d-1+\prod_{i=1}^{d-1} k_i.\]
\end{proof}

In the next result, we obtain an upper bound on the number of minimally non-$O(n^{d-1})$ $d$-dimensional 0-1 matrices with first $d-1$ dimensions $k_1 \times k_2 \times \dots \times k_{d-1}$. This generalizes a result from \cite{Crowdmath2018}. 

\begin{thm}
Let
$$
S(k_1, k_2, \dots, k_{d-1}) = 1+2\sum_{j = 1}^{d-1}(2k_i-2)$$
and 
$$P(k_1, k_2, \dots, k_{d-1}) = \prod_{i = 1}^{d-1}k_i.$$
The number of minimally non-$O(n^{d-1})$ $d$-dimensional 0-1 matrices with first $d-1$ dimensions $k_1 \times k_2 \times \dots \times k_{d-1}$ is at most \[\sum_{j = 1}^{S(k_1, k_2, \dots, k_{d-1})} ((j+1)^{P(k_1, k_2, \dots, k_{d-1})}-j^{P(k_1, k_2, \dots, k_{d-1})})P(k_1, k_2, \dots, k_{d-1})^{j-1}.\]
\end{thm}

\begin{proof}
In a minimally non-$O(n^{d-1})$ $d$-dimensional 0-1 matrix of dimensions $k_1 \times k_2 \times \dots \times k_{d-1} \times j$, there are at most $(j+1)^{P(k_1, k_2, \dots, k_{d-1})}-j^{P(k_1, k_2, \dots, k_{d-1})}$ combinations of first $1$-entries that can be deleted in each $d$-row. Indeed, $P(k_1, k_2, \dots, k_{d-1})$ counts the number of $d$-rows, $(j+1)^{P(k_1, k_2, \dots, k_{d-1})}$ counts the number of ways to choose a $d$-layer which contains the first $1$-entry or to not have a $1$-entry for each $d$-row, and $j^{P(k_1, k_2, \dots, k_{d-1})}$ counts the number of ways to choose a $d$-layer among the last $j-1$ $d$-layers which contains the first $1$-entry or to not have a $1$-entry for each $d$-row. Having all of the first $1$-entries in the last $j-1$ $d$-layers would imply that the first $d$-layer has all zeroes, which is impossible by Lemma~\ref{lem:attach1}. After the first $1$-entries are deleted in each $d$-row, the first $d$-layer has no $1$-entries and each $d$-layer except the first has at most a single $1$-entry, since otherwise there would be some $2$-dimensional projection to some two dimensions including dimension $d$ which is equal to the $2 \times 2$ all-ones matrix,
$$
\begin{pmatrix}
1 & 1 & 0 \\
1 & 0 & 1
\end{pmatrix},
$$ or one of its reflections or rotations. If a $d$-layer has no $1$-entry removed, then there are at most $P(k_1, k_2, \dots, k_{d-1})$ possibilities for the location of its $1$-entry. Otherwise if the $d$-layer has some $1$-entry removed, then after that removal it either has no $1$-entries or it has some $1$-entry in a location different from the removed $1$-entry, so there are at most $P(k_1, k_2, \dots, k_{d-1})$ possibilities for the entries in the $d$-layer. In either case, every $d$-layer except the first has at most at most $P(k_1, k_2, \dots, k_{d-1})$ possibilities for its entries, so there are at most $P(k_1, k_2, \dots, k_{d-1})^{j-1}$ possible matrices. The upper bound of $S(k_1, k_2, \dots, k_{d-1})$ in the sum follows from Theorem~\ref{thm:minnonlin_kd}.
\end{proof}

For every positive integer $n$, Tardos constructed a family $\mathcal{J}_n\cup\mathcal{D}_n$ of $n\times n$ 0-1 matrices of weight $n$ \cite{Tardos2005} with the property that $\ex(n,\mathcal{J}_{n_0}\cup\mathcal{D}_{n_0})=O(1)$ for any fixed positive integer $n_0$. For any family $\mathcal{P}$ of 0-1 matrices, either $\mathcal{J}_{n_0}\cup\mathcal{D}_{n_0}$ has an element avoiding all elements of $\mathcal{P}$ for any positive integer $n_0$, in which case $\ex(n,\mathcal{P})=\Omega(n)$, or there exists a positive integer $n_0$ such that every element of $\mathcal{J}_{n_0}\cup\mathcal{D}_{n_0}$ contains some element of $\mathcal{P}$, implying that $\ex(n,\mathcal{P})=O(1)$.
Before generalizing to $d$-dimensional 0-1 matrices, we first give some definitions. The $d$-dimensional identity matrix $I_{n_0,d}$ is a $d$-dimensional 0-1 matrix of dimensions $n_0\times n_0\times\dots\times n_0$ where an entry is a $1$-entry if and only if all of its coordinates are equal. An equivalent $I'$ of $I_{n_0,d}$ is a permutation matrix of the same dimensions as $I_{n_0,d}$ such that if we order the $1$-entries by their first coordinates, then for every integer $i\in[2,d]$ their $i^{\text{th}}$ coordinates are either increasing or decreasing. For each of these $2^{d-1}$ equivalents we can define a unique partial order among entries in any $d$-dimensional matrix $M$, such that two entries $M(x_1,x_2,\dots,x_d)$ and $M(y_1,y_2,\dots,y_d)$ are comparable if for every integer $i\in[2,d]$ the quantity $(x_1-y_1)(x_i-y_i)$ is positive when the $i^{\text{th}}$ coordinates of the $1$-entries of $I'$ ordered by their first coordinates are increasing and negative when the $i^{\text{th}}$ coordinates are decreasing.

\begin{lem}\label{lem:JD}
Let $\mathcal{D}_{n_0,d}$ denote the set of $2^{d-1}$ equivalents of the $d$-dimensional identity matrix $I_{n_0,d}$, and let $\mathcal{J}_{n_0,d}$ denote the set of all $d$-dimensional $n_0\times n_0\times\ldots\times n_0$ 0-1 matrices with exactly $n_0$ $1$-entries such that every pair of $1$-entries have the same coordinate in at least one dimension. For any family $\mathcal{P}$ of $d$-dimensional 0-1 matrices, if there exists $n_0\in\mathbb{N}$ such that no element of $\mathcal{J}_{n_0,d}\cup\mathcal{D}_{n_0,d}$ avoids all patterns in $\mathcal{P}$, then $\ex(n,\mathcal{P},d)=O(1)$. Otherwise $\ex(n,\mathcal{P},d)\geq n$.
\end{lem}
\begin{proof}
If for every $n_0\in\mathbb{N}$ there is an element of $\mathcal{J}_{n_0,d}\cup\mathcal{D}_{n_0,d}$ that avoids all patterns in $\mathcal{P}$, then by definition $\ex(n,\mathcal{P},d)\geq n$ since every element of $\mathcal{J}_{n_0,d}\cup\mathcal{D}_{n_0,d}$ has dimension $n_0\times n_0\times\ldots\times n_0$ and $n_0$ $1$-entries.

Otherwise,
we show that $\ex(n,\mathcal{P},d)\leq (n_0-1)^{1+2^{d-1}}$. Let $M$ be a
matrix of dimensions $n\times\dots\times n$ with more than $(n_0-1)^{1+2^{d-1}}$ $1$-entries which avoids every matrix from $\mathcal{D}_{n_0,d}$. Under each of the $2^{d-1}$ possible partial orders, these $1$-entries do not contain a chain of length $n_0$. By Dilworth's theorem, this implies the existence of a partition of the $1$-entries into at most $n_0-1$ antichains. Repeating the application of Dilworth's theorem to every possible partial order, there is a partition of the $1$-entries into $(n_0-1)^{2^{d-1}}$ disjoint subsets where every pair of $1$-entries from the same subset are not comparable, i.e., have the same coordinate in at least one dimension. By pigeonhole principle one of these subsets has at least $n_0$ $1$-entries. Thus, $M$ contains some matrix in $\mathcal{J}_{n_0,d}$ and therefore also some matrix in $\mathcal{P}$.
\end{proof}

Below we determine the unique $d$-dimensional 0-1 matrix of dimensions $n\times\dots\times n$ that is $P$-saturated if $P$ has a single $1$-entry. Its unique structure will be used later to derive a result for extremal functions of families of $d$-dimensional 0-1 matrices.

\begin{prop}\label{prop:single_one_ddim}
If $P$ is a $d$-dimensional 0-1 matrix of dimensions $k_1 \times k_2 \times \dots \times k_d$ with a single $1$-entry with coordinates $(q_1,\dots,q_d)$ and $n\ge \max_{i\in[d]}k_i$, then $\ex(n, P, d) = \sat(n, P, d) = n^d - \prod_{i = 1}^d (n+1-k_i)$ and the unique  $P$-saturated $d$-dimensional 0-1 matrix $A$ of dimensions $n \times n \times \dots \times n$ has $1$-entries at any position with coordinates $(y_1,\dots,y_d)$ where $y_i\notin[q_i,n-k_i+q_i]$ for some integer $i\in[d]$ and $0$-entries elsewhere. Therefore for every integer $i\in[d]$ there is some $x_i\in[n]$ and a $P$-free $d$-dimensional 0-1 matrix of dimensions $n\times\dots\times n$ with $1$-entries at any position with $i^{\text{th}}$ coordinate equal to $x_i$ and $0$-entries elsewhere.
\end{prop}

\begin{proof}
For any $P$-saturated $d$-dimensional 0-1 matrix $A$ of dimensions $n \times n \times \dots \times n$, $A$ cannot have a $1$-entry at the position with coordinates $(y_1, y_2, \dots, y_d)$ if $y_i \in [q_i, n-k_i+q_i]$ for all integers $i\in[d]$ because that would make $A$ contain $P$. Also, $A$ must have a $1$-entry at any position with coordinates $(y_1, y_2, \dots, y_d)$ for which $y_i \not \in [q_i, n-k_i+q_i]$ for some integer $i\in[d]$ because flipping the $0$-entry at the position does not introduce $P$ in $A$. Thus, $A$ has the stated structure with weight $n^d - \prod_{i = 1}^d (n+1-k_i)$. For any integer $i\in[d]$, we can choose any $x_i$ from $[n]\setminus[q_i,n-k_i+q_i]$ and the stated $d$-dimensional 0-1 matrix of dimensions $n\times\dots\times n$ is a submatrix of $A$ and therefore is $P$-free.
\end{proof}

Using Lemma~\ref{lem:JD}, we demonstrate that for any integer $k\in[0,d-1]$, no family of $d$-dimensional 0-1 matrices of size $k$ has extremal function in $(0, n^{d-k})$.
\begin{prop}\label{prop:either_or_prop}
If $\mathcal{P}$ is a family of $d$-dimensional 0-1 matrices with $k$ members for some $k < d$, then either $\ex(n, \mathcal{P}, d) = 0$ or $\ex(n, \mathcal{P}, d) \ge n^{d-k}$. 
\end{prop}

\begin{proof}
Suppose that $\ex(n, \mathcal{P}, d) \neq 0$. Then there exists a $\mathcal{P}$-free $d$-dimensional 0-1 matrix $A$ of dimensions $n \times n \times \dots \times n$ with at least one 1-entry. Suppose that this 1-entry has coordinates $(x_1, \dots, x_d)$. Note that $\mathcal{P}$ cannot include any pattern with all entries equal to zero, or else $\ex(n, \mathcal{P}, d) = 0$. Moreover, $\mathcal{P}$ cannot include the $1\times1\times\dots\times 1$ matrix with a single $1$-entry, or else again we would have $\ex(n, \mathcal{P}, d) = 0$. Thus all patterns in $\mathcal{P}$ must have some dimension with length at least $2$, and they all must have at least one $1$-entry.

For any pattern $P \in \mathcal{P}$ with at least two $1$-entries, there exists an integer $i\in[d]$ such that the $d$-dimensional 0-1 matrix $B_P$ of dimensions $n \times n \times \dots \times n$ with $1$-entries at every entry with $i^{\text{th}}$ coordinate equal to some $x_i\in[n]$ is $P$-free (simply choose $i$ such that $P$ has at least two $1$-entries with different $i^{\text{th}}$ coordinates). For any pattern $P \in \mathcal{P}$ with only a single $1$-entry, it must have some dimension with length at least $2$. By Proposition~\ref{prop:single_one_ddim} there exists $x_i\in[n]$ for each dimension $i\in[d]$ on which the side length of $P$ is at least $2$, such that the $d$-dimensional 0-1 matrix $B_P$ of dimensions $n \times n \times \dots \times n$ with $1$-entries at every entry with $i^{\text{th}}$ coordinate equal to $x_i$, whenever it is defined, is $P$-free. Consider a $d$-dimensional 0-1 matrix $B$ of dimensions $n \times n \times \dots \times n$ such that each entry of $B$ is a $1$-entry if and only if for every $P \in \mathcal{P}$ the corresponding entry in $B_P$ is a $1$-entry. Then $B$ is $\mathcal{P}$-free and has at least $n^{d-k}$ $1$-entries. Therefore, $\ex(n, \mathcal{P}, d) \ge n^{d-k}$. 
\end{proof}

Given the degree of freedom that we have when constructing families of $d$-dimensional 0-1 matrices, it is possible to construct a family of $d$-dimensional 0-1 matrices with extremal function equal to $kn^r$ for any positive integer $k$ and integer $r\in[0,d-1]$.
\begin{thm}\label{thm:knr}
For all positive integers $d$ and $k$ and every integer $r \in [0, d-1]$, there exists a family of 0-1 matrices $\mathcal{P}_{d,k,r}$ such that $\ex(n, \mathcal{P}_{d,k,r}, d) = k n^r$ for all $n$ sufficiently large. Moreover, the smallest such $\mathcal{P}_{d,k,r}$ has size $|\mathcal{P}_{d,k,r}| = d-r$.
\end{thm}

\begin{proof}
To see that there exists such a $\mathcal{P}_{d,k,r}$ with $|\mathcal{P}_{d,k,r}| = d-r$, we define $d$ forbidden patterns. First we define $d-1$ patterns $P_2, \dots, P_d$ where dimension $i$ of $P_i$ has length $2$, all other dimensions of $P_i$ have length $1$, and $P_i$ consists of a 0-entry with all coordinates equal to $1$ and a 1-entry with the $i^{\text{th}}$ coordinate equal to $2$ and all other coordinates equal to $1$. Next we define the pattern $Q$ where the first dimension has length $k+1$, all other dimensions have length $1$, and all $k+1$ entries of $Q$ are equal to $1$. Consider the family $\mathcal{P}_{d,k,r}$ which consists of the $d-1-r$ patterns $P_2, \dots, P_{d-r}$ and the pattern $Q$ when $r < d-1$, and only the pattern $Q$ when $r = d-1$. Clearly $\ex(n, \mathcal{P}_{d,k,d-1}, d) = k n^{d-1}$, so we assume for the rest of the proof that $r < d-1$.

Suppose that $A$ is a $\mathcal{P}_{d,k,r}$-free $d$-dimensional 0-1 matrix with all dimensions of length $n$, where $n \ge k$. Since $A$ avoids $P_i$ for each $2 \le i \le d-r$, $A$ cannot have any $1$-entry with $i^{\text{th}}$ coordinate greater than $1$. Thus the only entries where $A$ could possibly have a $1$-entry are the $n^{r+1}$ entries whose $2^{\text{nd}}$ through $(d-r)^{\text{th}}$ coordinates are all equal to $1$. For any fixed values of the last $r$ coordinates, there are $n$ such entries, and at most $k$ of those entries can be $1$-entries, or else $A$ would contain $Q$. Thus $A$ has at most $k n^r$ $1$-entries, so $\ex(n, \mathcal{P}_{d,k, r}, d) \le k n^r$. Moreover, consider any $d$-dimensional 0-1 matrix $B$ with all dimensions of length $n \ge k$ and all entries equal to $0$ except for the $k n^r$ entries whose $2^{\text{nd}}$ through $(d-r)^{\text{th}}$ coordinates are all ones and whose first coordinates are in $[k]$. By construction, $B$ avoids every pattern in the family $\mathcal{P}_{d,k, r}$. Thus, $\ex(n, \mathcal{P}_{d,k, r}, d) = k n^r$.

For any family $\mathcal{P}$ with $|\mathcal{P}| < d-r$, we must have either $\ex(n, \mathcal{P}, d) = 0$ or $\ex(n, \mathcal{P}, d) \ge n^{d-|\mathcal{P}|}$ by Proposition~\ref{prop:either_or_prop}. Thus, $\ex(n, \mathcal{P}, d) \neq k n^r$ for all $n$ sufficiently large.
\end{proof}

\section{Saturation and semisaturation functions of $d$-dimensional 0-1 matrices}\label{sec:sat}

We make a simple observation below about the semisaturation function of general patterns. The saturation function does not have the property.
\begin{observation}
If $\mathcal{P}$ and $\mathcal{Q}$ are two families of patterns and $\mathcal{P}\subset\mathcal{Q}$, then
$$
\ssat(n,\mathcal{Q})\leq \ssat(n,\mathcal{P})
$$
\end{observation}
\begin{proof}
If a pattern $M$ is $\mathcal{P}$-semisaturated, then it is also $\mathcal{Q}$-semisaturated because whenever a new copy of some element of $\mathcal{P}$ is introduced in $M$ there is also a new copy of some element of $\mathcal{Q}$ that is introduced in $M$.
\end{proof}

The following is extended from Lemma 4.5 in \cite{Tsai2023} and will be used later to prove more results about semisaturation functions.
\begin{lem}\label{lem:only}
Let $\mathcal{P}$ be a family of non-zero $d$-dimensional 0-1 matrices. Suppose that $d'<d$. 
If no $P\in\mathcal{P}$ contains a $1$-entry which is the only $1$-entry in every $d'$-dimensional cross section of $P$ that it belongs to, then $\ssat(n,\mathcal{P},d)=\Omega(n^{d-d'})$.
\end{lem}
\begin{proof}
Suppose that $M$ is $\mathcal{P}$-semisaturated with each dimension of length $n$. Say that a $0$-entry and a $1$-entry of $M$ are \textit{connected} if they are in the same $d'$-dimensional cross section of $M$. Each $0$-entry is connected with at least one $1$-entry, and each $1$-entry is connected with at most $\binom{d}{d'}(n^{d'}-1)$ $0$-entries. So the weight of $M$ is at least 
$$
\frac{n^d}{1+\binom{d}{d'}(n^{d'}-1)}=\Theta(n^{d-d'}).
$$
\end{proof}

Before presenting our main result about semisaturation function of families of $d$-dimensional 0-1 matrices, we describe the existing necessary and sufficient conditions for a $d$-dimensional 0-1 matrix to have a bounded semisaturation function.
\begin{thm}\label{thm:bounded-ssat}
\cite{Tsai2023}
Given a non-empty $d$-dimensional 0-1 matrix $P$, $\ssat(n,P,d)=O(1)$ if and only if both of the following properties hold for $P$:\\
(i) For any integer $d'\in[d-1]$, every $d'$-dimensional face $f$ of $P$ contains a $1$-entry $o$ that is the only $1$-entry in every $(d-1)$-dimensional cross section that is orthogonal to $f$ and contains $o$.\\
(ii) $P$ contains a $1$-entry that is the only $1$-entry in every $(d-1)$-dimensional cross section that it belongs to.\\
	Otherwise, if at least one of the properties does not hold for $P$, then $\ssat(n,P,d)=\Omega(n)$.
\end{thm}

We show that every family of $d$-dimensional 0-1 matrices must have semisaturation function $\Theta(n^k)$ for some integer $k\in[0,d-1]$, and we give the necessary and sufficient conditions for a family of $d$-dimensional 0-1 matrices to have a $\Theta(n^{k})$ semisaturation function.

Because we do not want to mix up faces of different matrices, we use the notation of counterpart in the statement and proof. Moreover, for consistency of distinguishing which side of the matrix a face is on, we refer each face as the counterpart of some face of a $d$-dimensional 0-1 matrix of dimensions $2\times\dots\times2$. In this way, regardless of the dimensions of the matrix whose face we are interested in, the fixed coordinates of its counterpart in a $d$-dimensional 0-1 matrix of dimensions $2\times\dots\times2$ is always either $1$ or $2$.

\begin{thm}\label{thm:ssat}
Given a family $\mathcal{P}$ of non-empty $d$-dimensional 0-1 matrices, there exists an integer $k\in[0, d-1]$ such that $\ssat(n,\mathcal{P},d)=\Theta(n^k)$. In particular, $k$ is the smallest integer in $[0, d-1]$ such that both properties below hold for $(\mathcal{P},k)$:\\
(i) For every integer $d'\in[k+1,d-1]$ and for every $d'$-dimensional face $f$ of a $d$-dimensional $2\times\dots\times 2$ 0-1 matrix, there exists $P\in\mathcal{P}$ that has a $1$-entry $o$ in $f$'s counterpart $f_P$ in $P$ that is the only $1$-entry in every cross section that is $(k+1)$-orthogonal to $f_P$ and contains $o$.\\
(ii) There exists $P\in\mathcal{P}$ that contains a $1$-entry that is the only $1$-entry in every $(d-1-k)$-dimensional cross section that it belongs to.
\end{thm}
\begin{proof}
Let $k$ be the smallest integer in $[0,d-1]$ such that both properties hold for $(\mathcal{P},k)$. We start by proving that $\ssat(n,\mathcal{P},d)=\Omega(n^k).$ Specifically, since $k$ is the smallest integer in $[0,d-1]$ such that both properties hold for $(\mathcal{P},k)$, either
property (i) or (ii) does not hold for $(\mathcal{P},k-1)$. Below, we prove that $\ssat(n,\mathcal{P},d)=\Omega(n^k)$ in either case.

Let $M$ be a $d$-dimensional $\mathcal{P}$-semisaturated 0-1 matrix with all dimensions of length $n$. The statement holds trivially for $k=0$, so assume that $k>0$. Suppose that property (i) does not hold for $(\mathcal{P},k-1)$, i.e., there exists a $d'$-dimensional face $f$ of a $d$-dimensional $2\times\dots\times 2$ 0-1 matrix where $d'\in[k,d-1]$, such that no $P\in\mathcal{P}$ contains a $1$-entry $o$ in $f$'s counterpart $f_P$ in $P$ that is the only $1$-entry in every cross section which is $k$-orthogonal to $f_P$ and contains $o$. Let the counterpart of $f$ in $M$ be $f_M$. If $f_M$ has all entries equal to $1$, then $w(M)=\Omega(n^{k})$. Otherwise, we say that a $0$-entry in $f_M$ is connected to a $1$-entry in $M\setminus f_M$ if they are in the same cross section that is $k$-orthogonal to $f_M$. Each $0$-entry of $f_M$ is connected to at least one $1$-entry. Each $1$-entry of $M\setminus f_M$ is in $(2^{d-d'}-1)\sum_{i=k}^{d'}\binom{d'}{i}$ cross sections that are $k$-orthogonal to $f_M$, and each of them contains at most $n^{d'-k}$ $0$-entries of $f_M$. Thus each $1$-entry of $M\setminus f_M$ is connected to at most $(2^{d-d'}-1)\sum_{i=k}^{d'}\binom{d'}{i}n^{d'-k}$ $0$-entries in $f_M$. If $f_M$ has $\alpha$ $0$-entries, then
\begin{equation}
\begin{split}
w(M)&\geq (n^{d'}-\alpha)+\frac{\alpha}{(2^{d-d'}-1)\sum_{i=k}^{d'}\binom{d'}{i}n^{d'-k}}\\
&\geq \frac{n^{d'}}{(2^{d-d'}-1)\sum_{i=k}^{d'}\binom{d'}{i}n^{d'-k}}\\
&=\Theta(n^{k}).
\end{split}
\end{equation}
Suppose instead that property (ii) does not hold for $(\mathcal{P},k-1)$, i.e., no $P\in\mathcal{P}$ has a $1$-entry that is the only $1$-entry in every $(d-k)$-dimensional cross section that it belongs to. By Lemma~\ref{lem:only} $\ssat(n,\mathcal{P},d)=\Omega(n^{k})$.

Now, we prove that $\ssat(n,\mathcal{P},d)=O(n^k).$ In order to do so, we break down the remaining proof into three steps. In Step 1, 
we construct an $n\times\dots\times n$ 0-1 matrix $M$
  with $O(n^k)$ $1$-entries. In subsequent steps, depending on the position of an arbitrary flipped $0$-entry $M(z_1,\dots,z_d)$, we show that the flipped $0$-entry introduces a new copy of some $P\in\mathcal{P}$ and thus $M$ is $\mathcal{P}$-semisaturated.
In Step 2a we specify a 0-1 matrix $P\in\mathcal{P}$ and a $1$-entry $o$ of $P$.
In Step 2b we specify a submatrix $P'$ of $M'$, the flipped $M$, where $P'$ contains the flipped entry and has the same dimensions as $P$.
  In Step 2c, we show that $P'$ contains $P$ by checking that each entry of $P'$ that is not flipped is either a $1$-entry by construction of $M$ or corresponds to a $0$-entry of $P$ so its value does not affect whether $P'$ contains $P$. Moreover the flipped entry corresponds to $o$. Steps 3a-3c are similar to Steps 2a-2c except that $P',P,o$ are specified differently because the flipped $0$-entry is different by nature. With these steps we conclude that flipping any $0$-entry of $M$ introduces a new copy of some $P\in\mathcal{P}$, so $M$ is $\mathcal{P}$-semisaturated.

\noindent
\textbf{(1)} Suppose that the dimensions of any given $P\in\mathcal{P}$ are $l_{P,1}\times\dots\times l_{P,d}$ and for every integer $i\in[d]$, denote $\max_{P\in\mathcal{P}}l_{P,i}$ by $l_i$.
Let $M(x_1,\dots,x_d)=1$ if and only if there are at least $d-k$ choices of $i$ in $[d]$ such that $x_i$ is outside $[l_i,n+1-l_i]$.

\noindent
\textbf{(2)} We take
 this step when $z_i\in[l_{i},n+1-l_{i}]$ for each integer $i\in[d]$.

\noindent
\textbf{(2a)}  Let $o=P(o_1,\dots,o_d)$ be a $1$-entry of some $P\in\mathcal{P}$ with property (ii).

\noindent
\textbf{(2b)} In each dimension $i$ we restrict $M'$ to the indices
$[1,o_i-1]\cup\{z_i\}\cup[n-l_{P,i}+o_i+1,n]\}$
to obtain a submatrix $P'$ of the same dimensions as $P$ that contains the flipped $1$-entry.

\noindent
\textbf{(2c)} Consider an entry $p'=P'(z'_1,\dots,z'_d)$. If $z'_i\neq o_i$ for at least $d-k$ choices of $i$, then $p'=1$ by construction of $M$. If $z'_i=o_i$ for every integer $i\in[d]$, then $p'$ is the flipped $1$-entry matching $o$. Otherwise consider a cross section $g$ of $P$ such that its set of fixed dimensions $C_g$ is some $(k+1)$-subset of $\{i|i\in[d],z'_i=o_i\}$ and for each $i \in C_g$, coordinate $i$ of $g$ has the fixed value $o_i$. Cross section $g$ contains both $o$ and the counterpart of $p'$ in $P$. By property (ii), $g$ does not contain any $1$-entry other than $o$ and thus whether $P'$ contains $P$ does not depend on the value of $p'$.

\noindent
\textbf{(3)} We take this step when
$z_i\notin [l_{i},n+1-l_i]$ for some integer $i\in[d]$. We split $[d]$ into disjoint sets $X,Y,Z$: 
\begin{equation}
\begin{split}
X&=\{i|i\in[d],z_i<l_i\}\\
Y&=\{i|i\in[d],n+1-l_i<z_i\}\\
Z&=[d]\setminus X\setminus Y\\
\end{split}
\end{equation}
By assumption $X\cup Y\neq \emptyset$. Moreover, given that $M(z_1,\dots,z_d)=0$ and the the way that $M$ is constructed, we have $k<|Z|$.

Let $f$ be the face of a $d$-dimensional $2\times\ldots\times 2$ 0-1 matrix such that the $i^{\text{th}}$ coordinate of $f$ is fixed to $1$ or $2$ if $i\in X$ or $i\in Y$, respectively. In other words, $C_f=X\cup Y$ and the dimensionality of $f$ is $|Z|\in[k+1,d-1]$.

\noindent
\textbf{(3a)} Let $P\in\mathcal{P}$ be the $d$-dimensional 0-1 matrix with property (i) and let $f_P$ be $f$'s counterpart in $P$. For each $i\in Z$ let $l_i'$ be the $i^{\text{th}}$ coordinate of the $1$-entry $o$ contained in $f_P$ with property (i).

\noindent
\textbf{(3b)} We restrict $M'$ to the indices $I_i$ below for each $i\in [d]$ to obtain a submatrix $P'$ of the same dimensions of $P$ that contains the flipped $1$-entry:
$$
I_i=   
\begin{cases}
\{z_i\}\cup[n-l_{P,i}+2,n], & \text{if }i\in X\\
[1,l_{P,i}-1]\cup\{z_i\}, & \text{if }i\in Y\\
[1,l_{i}'-1]\cup\{z_i\}\cup[n-l_{P,i}+l_{i}'+1,n], &\text{otherwise}
\end{cases}.
$$

\noindent
\textbf{(3c)} Consider an entry $p'=P'(z'_1,\ldots,z'_d)$. Based on $p'$ we further split $X,Y,Z$:
\begin{equation}
\begin{split}
X_1&=\{i|z'_i=1\}\cap X, X_2=X\setminus X_1\\
Y_1&=\{i|z'_i=l_{P,i}\}\cap Y, Y_2=Y\setminus Y_1\\
Z_1&=\{i|z'_i=l_i'\}\cap Z, Z_2=Z\setminus Z_1\\
\end{split}
\end{equation}
If $|Z_1|\leq k$, then $p'=1$ by construction of $M$. Suppose that $|Z_1|>k$.
If $X_2\cup Y_2\cup Z_2$ is empty, then $p'$ is the flipped $1$-entry matching $o$.
If $X_2\cup Y_2\cup Z_2$ is not empty, consider a cross section $g$ of $P$ with $C_g=Z_1$ and for every $i\in C_g$ the $i^\text{th}$ coordinate of every entry in $g$ is fixed to $l'_i$. We have that $g$ is $(k+1)$-orthogonal to $f_P$ and contains both $o$ and the counterpart of $p'$ in $P$. Since by property (i) $g$ has no $1$-entry other than $o$, whether $P'$ contains $P$ or not does not depend on the value of $p'$.
\end{proof}
Theorem~\ref{thm:ssat} aligns with the expectation that the semisaturation function of $I_{1,d}$, the $d$-dimensional identity matrix with a single $1$-entry, is zero. The following lemma could also be proved without Theorem~\ref{thm:ssat}, nonetheless we prove it using Theorem~\ref{thm:ssat}.
\begin{lem}
If a non-empty $d$-dimensional 0-1 matrix $P$ has $l<d$ dimensions with side length one, then its semisaturation function is $\Theta(n^k)$ for some integer $k\in[l, d-1]$.
\end{lem}
\begin{proof}
The statement holds trivially when $l=0$, so assume that $0<l$.
It suffices to show that either property (i) or property (ii) does not hold for $(P,l-1)$. Suppose that the side lengths of the first $l$ dimensions are all ones, and property (ii) holds for $(P,l-1)$. Consider a cross section $j$ containing a $1$-entry $o$ in property (ii) with $C_j=[l]$. By property (ii) $o$ is the only $1$-entry in $j$, and it is also the only $1$-entry in $P$. Consider two $l$-dimensional faces $g$ and $h$ of $P$ where $(l+1)\in C_g\cap C_h$ and their fixed $(l+1)^{\text{th}}$ coordinates are different. Hence at least one of the two faces is empty, and property (i) does not hold of $(P,l-1)$.
\end{proof}

For every positive integer $d$ and every integer $k\in[0,d-1]$, we construct a $d$-dimensional 0-1 matrix with semisaturation function $\Theta(n^k)$.
\begin{lem}
For every $d\geq 2$ and $k\in[0,d-1]$, there exists a $d$-dimensional 0-1 matrix $P$ such that $\ssat(n,P,d)=\Theta(n^k)$.
\end{lem}
\begin{proof}
For a $d$-dimensional $l_1\times\dots l_d$ 0-1 matrix $P$ and a face $f$ of $P$, we say an entry $P(x_1,\dots,x_d)$ is \textit{interior} to $f$ if $1<x_i<l_i$ for 
  every $i\notin C_f$. Denote by property (iii) a slightly stronger property than property (i) in Theorem~\ref{thm:ssat}: a face $f$ has property (iii) if it has an interior $1$-entry $o$ such that every $(d-1)$-dimensional cross section which is orthogonal to $f$ and contains $o$ does not have any other $1$-entry.\\

Refer to properties (i) and (ii) in Theorem~\ref{thm:ssat}. In the context of a fixed $k$ and the constructed 0-1 matrix $P$, by slight abuse of notation when we say that property (i) holds for a face $f$
it means that $P$ has a $1$-entry $o$ in $f$'s counterpart $f_P$ that is the only $1$-entry in every cross section that is $(k+1)$-orthogonal to $f_P$ and contains $o$.
When we say that property (ii) holds for $(\{P\},k)$ it means that $P$ has a $1$-entry that is the only $1$-entry in every $(d-1-k)$-dimensional cross section that it belongs to.

We prove a stronger statement, that there exists a $d$-dimensional $l_1\times\ldots\times l_d$ 0-1 matrix $P$ such that
\begin{itemize}
    \item Property (i) does \textit{not} hold for any face of dimensionality less than $k+1$,
    \item Property (ii) holds for $(\{P\},0)$, and 
    \item Property (iii) holds
for every $d'$-dimensional face $f$ of $P$ where $d'\in[k+1, d-1]$.
\end{itemize}

We claim that it is indeed stronger than as stated in the result. When the third bullet holds, every face $f$ of dimensionality greater than $k$ has an interior $1$-entry $o$ such that every $(d-1)$-dimensional cross section $c$ that is orthogonal to $f$ and contains $o$ does not have any other $1$-entry. Since every cross section $c'$ that is $(k+1)$-orthogonal to $f$ and contains $o$ is contained in some $(d-1)$-dimensional cross section $c$ stated above, $c'$ also does not have any other $1$-entry. Therefore property (i) holds \textit{only} for faces of dimensionality greater than $k$.

\noindent
\textbf{Overview of the remaining proof}
In Step 1, we begin with an all-zeros matrix $P$. For each face of $P$ of dimensionality greater than $k$ for which property (iii) does not hold, we insert an interior $1$-entry in it to make property (iii) hold for it. In Step 1a we show that by the ordering of these insertions, no face loses property (iii) later in this step. In step 1b we show that the insertion does not make property (i) hold for any face of dimensionality less than $k+1$. In Step 2, when property (iii) holds only for every face of dimensionality greater than $k$ and property (ii) does not hold for $(\mathcal{P},0)$ we insert a $1$-entry to $P$.
In Step 2a we show that this insertion does not make any face lose property (iii).
In Step 2b we show that the insertion does not make property (i) hold for any face of dimensionality less than $k+1$. Thus the stronger requirements are satisfied for the final $P$ and we are done.

\noindent
\textbf{(1)}
We begin with a $d$-dimensional $4\times 4\times\ldots\times 4$ empty 0-1 matrix $P$, i.e., we initially set all side lengths $l_1,l_2,\dots, l_d$ of $P$ to $4$. For $d'=k+1,\ldots,d-1$, we pick a $d'$-dimensional face $f$ of $P$ without an interior $1$-entry satisfying property (iii), and replace $P$ by $P'$ constructed below. Suppose that the coordinate of every entry in $f$ is fixed to some $b_i\in\{1,l_i\}$ for every $i\in C_f$. 
Denote the indicator function by $\mathbbm{1}()$, which evaluates to $1$ when the condition in the parenthesis holds and $0$ otherwise.
  For matrix $P'$ and any $i\in[d]$ the side length of $P$ on dimension $i$ is $l_i+\mathbbm{1}(i\notin C_f)$, and
\begin{enumerate}
\item $P'(x_1,\ldots,x_d)=P\left(x_1-\mathbbm{1}(1\notin C_f\wedge x_1>\lceil\frac{l_1}{2}\rceil),\ldots,x_d-\mathbbm{1}(d\notin C_f\wedge x_d>\lceil\frac{l_d}{2}\rceil)\right)$
if $x_i\ne \lceil \frac{l_i}{2}\rceil$ for every $i\notin C_f$.\\
\item $P'(x_1,\ldots,x_d)=1$ if $x_i=b_i$ for every $i\in C_f$ and $x_i=\lceil\frac{l_i}{2}\rceil$ for every $i\notin C_f$.\\
\item $P'(x_1,\ldots,x_d)=0$ otherwise.\\
\end{enumerate}
Note that the insertion in the second condition above adds an interior $1$-entry $o$  to the counterpart of $f$, which is the only $1$-entry in every $(d-1)$-dimensional cross section that is orthogonal to the counterpart of $f$ and contains $o$.

\noindent
\textbf{(1a)}
We prove that this insertion does not eliminate property (iii) for the counterpart of any other face $f'$ of $P$ in $P'$.
Suppose to the contrary that $f'$ is orthogonal to a $(d-1)$-dimensional cross section $g$, their intersection contains a $1$-entry interior to $f'$ before the insertion that is the only $1$-entry in $g$, and the insertion above adds $o$ to the counterpart of $g$. Let $C_g=\{j\}$. If $j \in C_f$, then since the counterpart of $g$ contains $o$, we have that the $j^{\text{th}}$ coordinate of $g$ must be fixed to $b_{j}\in\{1,l_{j}\}$, the $j^{\text{th}}$ coordinate of $f'$ is not fixed, and the $j^{\text{th}}$ coordinate of the intersection of $g$ and $f'$ is $b_{j}$. Thus the intersection of $f'$ and $g$ did not have a $1$-entry interior to $f'$, a contradiction. If $j \notin C_f$, then the $j^{\text{th}}$ coordinate of $g$ is fixed to $\lceil \frac{l_{j}}{2}\rceil$, and based on the rules above the counterpart of $g$ does not have any other $1$-entry. Therefore the insertion does not make any face $f'$ lose property (iii). 

\noindent
\textbf{(1b)}
No face of dimensionality less than $d'+1\le k+1$ contains the added $1$-entry $o$. So when completed no face of dimensionality less than $k+1$ has a $1$-entry, thus property (i) does not hold for any face of dimensionality less that $k+1$.

\noindent
\textbf{(2)}
After a sequence of insertions above, for each integer $d'\in[k+1,d-1]$ every $d'$-dimensional face $f$ of $P$ has property (iii). If property (ii) does not hold for $(\{P\},0)$, from $P$ we construct a $d$-dimensional 0-1 matrix $P'$ with the side length on every dimension further increased by $1$ as below.
$$
P'(x_1,\ldots,x_d)=
\begin{cases}
	1 & \forall i\in[d], x_i=\lceil l_i/2\rceil\\
	P(x_1-\mathbbm{1}(x_1>\lceil l_1/2\rceil),\ldots,x_d-\mathbbm{1}(x_d>\lceil l_d/2\rceil)) & \forall i\in[d], x_i\ne \lceil l_i/2\rceil\\
	0&\text{otherwise}\\
\end{cases}.
$$

\noindent
\textbf{(2a)}
The added $1$-entry is the only $1$-entry in every $(d-1)$-dimensional cross section it belongs to. Moreover it is not an interior $1$-entry of any face, so the insertion does not make any face of $P$ lose property (iii).

\noindent
\textbf{(2b)}
The argument is same as Step 1b.
No face of dimensionality less than $d'+1\le k+1$ contains the added $1$-entry $o$, so every such face still does not contain any $1$-entry. Thus property (i) does not hold for any face of dimensionality less that $k+1$.

\end{proof}

\begin{cor}
For every positive integer $d\geq 2$ and integer $k\in[0,d-1]$, there exists a family of $d$-dimensional 0-1 matrices $\mathcal{P}$ such that $\ssat(n,\mathcal{P},d)=\Theta(n^k)$.
\end{cor}

We generalize Theorem 1.3 of \cite{FK2021} below to show that no family of $d$-dimensional 0-1 matrices has saturation function strictly between $O(1)$ and $\Omega(n)$.
\begin{lem}
For any family $\mathcal{P}$ of non-empty $d$-dimensional 0-1 matrices, the function $\sat(n,\mathcal{P},d)$ is either $O(1)$ or $\Omega(n)$.
\end{lem}
\begin{proof}
Let $k=\max_{P\in\mathcal{P},i\in[d]}l_{P,i}$ be the maximum side length of all $d$-dimensional 0-1 matrices in $\mathcal{P}$, where the dimension of $P\in\mathcal{P}$ is $l_{P,1}\times l_{P,2}\times\ldots\times l_{P,d}$. If there exists a positive integer $n_0\in\mathbb{N}$ such that $\sat(n_0,\mathcal{P},d)<\frac{n_0}{k-1}$, then by definition there is a $d$-dimensional $n_0\times n_0\times\ldots\times n_0$ 0-1 matrix $M$ with less than $\frac{n_0}{k-1}$ $1$-entries that is $\mathcal{P}$-saturated, which has at least $k-1$ consecutive empty $i$-layers for every integer $i\in[d]$. Clearly $M$ is still $\mathcal{P}$-saturated if we insert any number of empty $i$-layers into the existing empty $i$-layers, so $\sat(n,\mathcal{P},d)=O(1)$. If for every positive integer $n_0\in\mathbb{N}$ we have $\sat(n_0,\mathcal{P},d)\geq\frac{n_0}{k-1}$, then $\sat(n,\mathcal{P},d)=\Omega(n)$.
\end{proof}

In the next result, we show that it is possible to construct a family of $d$-dimensional 0-1 matrices with saturation function equal to $k n^r$ for any positive integer $k$ and integer $r \in [0, d-1]$. This is an analogue of Theorem~\ref{thm:knr} which we proved for extremal functions. However, Theorem~\ref{thm:knr} also identified the minimum possible size of such a family for extremal functions. It remains to determine the minimum possible size of such a family for saturation functions.

\begin{thm}\label{thm:satknr}
For all positive integers $d$ and $k$ and every integer $r \in [0, d-1]$, there exists a family of 0-1 matrices $\mathcal{P}_{d,k,r}$ such that $\sat(n, \mathcal{P}_{d,k,r}, d) = k n^r$ for all $n$ sufficiently large.
\end{thm}

\begin{proof}
We use the same construction as in Theorem~\ref{thm:knr}. We define $d$ forbidden patterns. First we define $d-1$ patterns $P_2, \dots, P_d$ where dimension $i$ of $P_i$ has length $2$, all other dimensions of $P_i$ have length $1$, and $P_i$ consists of a 0-entry with all coordinates equal to $1$ and a 1-entry with the $i^{\text{th}}$ coordinate equal to $2$ and all other coordinates equal to $1$. Next we define the pattern $Q$ where the first dimension has length $k+1$, all other dimensions have length $1$, and all $k+1$ entries of $Q$ are equal to $1$. Consider the family $\mathcal{P}_{d,k,r}$ which consists of the $d-1-r$ patterns $P_2, \dots, P_{d-r}$ and the pattern $Q$ when $r < d-1$, and only the pattern $Q$ when $r = d-1$. Clearly $\sat(n, \mathcal{P}_{d,k,d-1}, d) = k n^{d-1}$ since any $d$-dimensional 0-1 matrix which is $Q$-saturated must have exactly $k$ ones in each $1$-row, so we assume for the rest of the proof that $r < d-1$.

By Theorem~\ref{thm:knr}, we immediately obtain $\sat(n, \mathcal{P}_{d,k, r}, d) \le k n^r$. For the lower bound, suppose that $A$ is a $\mathcal{P}_{d,k,r}$-saturated $d$-dimensional 0-1 matrix with all dimensions of length $n$, where $n \ge k$. Since $A$ avoids $P_i$ for each $2 \le i \le d-r$, $A$ cannot have any $1$-entry with $i^{\text{th}}$ coordinate greater than $1$. Thus the only entries where $A$ could possibly have a $1$-entry are the $n^{r+1}$ entries whose $2^{\text{nd}}$ through $(d-r)^{\text{th}}$ coordinates are all equal to $1$. For any fixed values of the last $r$ coordinates, there are $n$ such entries, and at most $k$ of those entries can be $1$-entries, or else $A$ would contain $Q$.  Since $A$ is $\mathcal{P}_{d,k,r}$-saturated, $A$ must have exactly $k$ ones in each $1$-row for which the $2^{\text{nd}}$ through $(d-r)^{\text{th}}$ coordinates are all equal to $1$, or else we could add a $1$-entry to $A$ in the same $1$-row and obtain a new $d$-dimensional 0-1 matrix which still avoids $\mathcal{P}_{d,k,r}$. Thus, $\sat(n, \mathcal{P}_{d,k, r}, d) \ge k n^r$, so we have $\sat(n, \mathcal{P}_{d,k,r}, d) = k n^r$.
\end{proof}

Next, we show that the ratio between extremal function and saturation function of families of $d$-dimensional 0-1 matrices can be almost as large as the maximum possible ratio of $n^d$.
\begin{prop}
For every $d \ge 2$ and $\epsilon > 0$, there exists a family of $d$-dimensional 0-1 matrices $\mathcal{P}$ such that $\sat(n,\mathcal{P},d) = O(1)$ and $\ex(n, \mathcal{P}, d) = \Omega(n^{d-\epsilon})$. 
\end{prop}

\begin{proof}
First, define $B_{d, r}$ to be the $d$-dimensional 0-1 matrix with all dimensions of length $r+1$, where the entries with all coordinates at most $r$ are $1$-entries and the entries with some coordinate equal to $r+1$ are $0$-entries. Let $\mathcal{P}_{d,r}$ be the family of $d$-dimensional 0-1 matrices consisting of $(r+1)^d - r^d$ forbidden patterns, where each pattern in $\mathcal{P}_{d,r}$ is obtained from $B_{d, r}$ by changing a single $0$-entry to $1$-entry.

Consider the $d$-dimensional 0-1 matrix $A$ with all dimensions of length $n > r$, where the entries with all coordinates at most $r$ are $1$-entries and the entries with some coordinate at least $r+1$ are $0$-entries. Note that changing any $0$-entry to an $1$-entry in $A$ produces a copy of some element of $\mathcal{P}_{d,r}$, so $A$ is $\mathcal{P}_{d,r}$-saturated. Thus $\sat(n,\mathcal{P}_{d,r},d) \le r^d = O(1)$. On the other hand, every pattern in $\mathcal{P}_{d,r}$ contains $B_{d, r}$, and \[\ex(n, B_{d, r}, d) = \Omega(n^{d-\frac{d(r-1)}{r^{d}-1}})\] by Theorem 2.1 of \cite{GT2017}, so \[\ex(n, \mathcal{P}_{d,r}, d) = \Omega(n^{d-\frac{d(r-1)}{r^{d}-1}}).\] Thus, we can choose $r$ sufficiently large so that $\frac{d(r-1)}{r^d-1} < \epsilon$.
\end{proof}

\section{Conclusion}\label{sec:con}

We identified several minimally non-$O(n^2)$ $3$-dimensional 0-1 matrices in Section~\ref{sec:ex}. It would be interesting to identify additional minimally non-$O(n^2)$ $3$-dimensional 0-1 matrices, and more generally to identify minimally non-$O(n^{d-1})$ $d$-dimensional 0-1 matrices for $d > 3$. One of the minimally non-$O(n^2)$ $3$-dimensional 0-1 matrices that we identified was
$$P=
\left(
			\begin{pmatrix}
			0 & 1 \\
			1 & 0 \\
			\end{pmatrix};
			\begin{pmatrix}
			1 & 0 \\
			0 & 1 \\
			\end{pmatrix}
			\right),$$
for which we showed that $\ex\left(n,P \right)=\Omega(n^{2.5})$ and $\ex\left(n,P \right)=O(n^{2.75})$. It remains to determine $\ex\left(n,P \right)$ up to a constant factor. For this problem, we conjecture that $\ex\left(n,P \right) = \Theta(n^{2.5})$.

In the same section, we proved for each $d > 2$ that there exist infinitely many minimally non-$O(n^{d-1})$ $d$-dimensional 0-1 matrices with all dimensions of length greater than $1$, generalizing a result of Geneson from \cite{Geneson2009} for the $d = 2$ case. Our result for $d > 2$ used a new operation on $d$-dimensional 0-1 matrices $P$ that produces a $(d+1)$-dimensional 0-1 matrix whose extremal function is on the order of $n$ times the extremal function of $P$. The operation uses the new dimension to produce a diagonal version of the original forbidden pattern. There are already several other known operations which can be performed on 0-1 matrices that provide sharp bounds on the extremal function of the new pattern, and most of these operations are easy to generalize to $d$-dimensional 0-1 matrices \cite{GHLNPW2019}. For the purpose of obtaining new bounds on extremal functions of forbidden $d$-dimensional 0-1 matrices, it would be useful to find other operations on $d$-dimensional 0-1 matrices which provide sharp bounds on the extremal function of the new pattern.

We showed for any minimally non-$O(n^{d-1})$ $d$-dimensional 0-1 matrix of dimensions $k_1 \times k_2 \times \dots \times k_d$ with $k_1 \le k_2 \le \dots \le k_d$ that $k_d \le 1+2 \sum_{i =1}^{d-1} (2k_i - 2)$. We believe that this bound is far from sharp, so it would be interesting to improve the bound or to identify some minimally non-$O(n^{d-1})$ $d$-dimensional 0-1 matrices for which the ratio of the length of the longest dimension to the sum of the lengths of the other dimensions is as large as possible. We also derived an upper bound on the weight of a minimally non-$O(n^{d-1})$ $d$-dimensional 0-1 matrix in terms of the lengths of its dimensions. In particular, suppose that $P$ is a minimally non-$O(n^{d-1})$ $d$-dimensional 0-1 matrix of dimensions $k_1 \times k_2 \times \dots \times k_d$ for which $P$ does not have all ones, two dimensions of length $2$, and all other dimensions of length $1$. We proved that the weight of $P$ is at most $k_d-1+\prod_{i=1}^{d-1} k_i$. This bound is attained by the 0-1 matrix 
$$
\begin{pmatrix}
1 & 1 & 0\\
1 & 0 & 1\\
\end{pmatrix}
$$
and its reflections, rotations, and embeddings in higher dimensions. However, we have not found any other minimally non-$O(n^{d-1})$ $d$-dimensional 0-1 matrices which attain this bound, so it remains to sharpen the bound as much as possible for 0-1 matrices where the longest dimension has length greater than $3$. Another open problem is to sharpen our upper bound on the number of minimally non-$O(n^{d-1})$ $d$-dimensional 0-1 matrices with first $d-1$ dimensions $k_1 \times k_2 \times \dots \times k_{d-1}$.

In Section~\ref{sec:sat}, we proved for every positive integer $d\geq 2$ and integer $r\in[0,d-1]$ that there exists a $d$-dimensional 0-1 matrix $P$ such that $\ssat(n,P,d)=\Theta(n^r)$. Moreover, we showed for every family $\mathcal{P}$ of non-empty $d$-dimensional 0-1 matrices that there exists an integer $r \in [0,d-1]$ such that $\ssat(n,\mathcal{P},d)=\Theta(n^r)$. For each $r \in [0, d-1]$, we characterized which families $\mathcal{P}$ of non-empty $d$-dimensional 0-1 matrices have semisaturation function $\ssat(n,\mathcal{P},d)=\Theta(n^r)$.

For the saturation function, we showed for all positive integers $d$ and $k$ and every integer $r \in [0, d-1]$ that there exists a family of $d$-dimensional 0-1 matrices $\mathcal{P}_{d,k,r}$ such that $\sat(n, \mathcal{P}_{d,k,r}, d) = k n^r$ for all $n$ sufficiently large. It would be interesting to find the minimum possible size of such a family. We also showed in the same section that for any family $\mathcal{P}$ of non-empty $d$-dimensional 0-1 matrices, the function $\sat(n,\mathcal{P},d)$ is either $O(1)$ or $\Omega(n)$. Combined with Tsai's result from \cite{Tsai2023} that $\sat(n, P, d) = O(n^{d-1})$ for every $d$-dimensional 0-1 matrix $P$, this subsumes the result of Fulek and Keszegh \cite{FK2021} for $d = 2$ that $\sat(n,P)$ is either $O(1)$ or $\Theta(n)$ for every 0-1 matrix $P$. We conjecture for any forbidden family $\mathcal{P}$ of $d$-dimensional 0-1 matrices that $\sat(n,\mathcal{P},d)=\Theta(n^r)$ for some integer $r \in [0, d-1]$. If the conjecture is true, a more difficult problem is to characterize which families $\mathcal{P}$ of non-empty $d$-dimensional 0-1 matrices have saturation function $\sat(n,\mathcal{P},d)=\Theta(n^r)$ for each $r \in [0,d-1]$. 



\section*{Acknowledgment}
The authors would like to thank the reviewers for their valuable comments that help improve the paper significantly.
Shen-Fu Tsai is supported by the Ministry of Science and Technology of Taiwan under grants MOST 111-2115-M-008-010-MY2 and MOST 113-2115-M-008-006-MY3.

\bibliographystyle{plain}
\bibliography{pattern-avoidance}

\begin{thebibliography}{10}

\bibitem{Berendsohn2021}
Benjamin~Aram Berendsohn.
\newblock An exact characterization of saturation for permutation matrices.
\newblock {\em arXiv preprint arXiv:2105.02210}, 2021.

\bibitem{BG1991}
Dan Bienstock and Ervin Gy\H{o}ri.
\newblock An extremal problem on sparse {$0$}-{$1$} matrices.
\newblock {\em SIAM J. Discrete Math.}, 4(1):17--27, 1991.

\bibitem{BC2021}
Richard~A. Brualdi and Lei Cao.
\newblock Pattern-avoiding {$(0,1)$}-matrices and bases of permutation
  matrices.
\newblock {\em Discrete Appl. Math.}, 304:196--211, 2021.

\bibitem{Crowdmath2018}
P.~A. CrowdMath.
\newblock Bounds on parameters of minimally non-linear patterns.
\newblock {\em Electron. J. Combin.}, 25(1):Paper No. 1.5, 11, 2018.

\bibitem{DS1965}
H.~Davenport and A.~Schinzel.
\newblock A combinatorial problem connected with differential equations.
\newblock {\em Amer. J. Math.}, 87:684--694, 1965.

\bibitem{EHM1964}
P.~Erd\H{o}s, A.~Hajnal, and J.~W. Moon.
\newblock A problem in graph theory.
\newblock {\em Amer. Math. Monthly}, 71:1107--1110, 1964.

\bibitem{EM1959}
P.~Erd\H{o}s and L.~Moser.
\newblock Problem 11.
\newblock {\em Canadian Math. Bulletin}, 2:43, 1959.

\bibitem{Fox2013}
Jacob Fox.
\newblock Stanley-{W}ilf limits are typically exponential.
\newblock {\em arXiv preprint arXiv:1310.8378}, 2013.

\bibitem{FK2021}
Radoslav Fulek and Bal\'{a}zs Keszegh.
\newblock Saturation problems about forbidden 0-1 submatrices.
\newblock {\em SIAM J. Discrete Math.}, 35(3):1964--1977, 2021.

\bibitem{furedi1990maximum}
Zolt\'{a}n F\"{u}redi.
\newblock The maximum number of unit distances in a convex {$n$}-gon.
\newblock {\em J. Combin. Theory Ser. A}, 55(2):316--320, 1990.

\bibitem{FH1992}
Zolt\'{a}n F\"{u}redi and P\'{e}ter Hajnal.
\newblock Davenport-{S}chinzel theory of matrices.
\newblock {\em Discrete Math.}, 103(3):233--251, 1992.

\bibitem{Geneson2009}
Jesse Geneson.
\newblock Extremal functions of forbidden double permutation matrices.
\newblock {\em J. Combin. Theory Ser. A}, 116(7):1235--1244, 2009.

\bibitem{Geneson2019}
Jesse Geneson.
\newblock Forbidden formations in multidimensional 0-1 matrices.
\newblock {\em Eur. J. Comb.}, 78:147--154, 2019.

\bibitem{Geneson2021}
Jesse Geneson.
\newblock Almost all permutation matrices have bounded saturation functions.
\newblock {\em Electron. J. Combin.}, 28(2):Paper No. 2.16, 13, 2021.

\bibitem{Geneson2021b}
Jesse Geneson.
\newblock A generalization of the {K}\"{o}v\'{a}ri-{S}\'{o}s-{T}ur\'{a}n
  theorem.
\newblock {\em Integers}, 21:1--12, 2021.

\bibitem{GHLNPW2019}
Jesse Geneson, Amber Holmes, Xujun Liu, Dana Neidinger, Yanitsa Pehova, and
  Isaac Wass.
\newblock Ramsey numbers of ordered graphs under graph operations.
\newblock {\em arXiv preprint arXiv:1902.00259}, 2019.

\bibitem{GT2017}
Jesse Geneson and Peter~M. Tian.
\newblock Extremal functions of forbidden multidimensional matrices.
\newblock {\em Discrete Math.}, 340(12):2769--2781, 2017.

\bibitem{GT2020}
Jesse Geneson and Shen-Fu Tsai.
\newblock Sharper bounds and structural results for minimally nonlinear 0-1
  matrices.
\newblock {\em Electron. J. Combin.}, 27(4):Paper No. 4.24, 8, 2020.

\bibitem{Hesterberg2012}
Adam Hesterberg.
\newblock Extremal functions of excluded tensor products of permutation
  matrices.
\newblock {\em Discrete Math.}, 312(10):1646--1649, 2012.

\bibitem{Keszegh2009}
Bal\'{a}zs Keszegh.
\newblock On linear forbidden submatrices.
\newblock {\em J. Combin. Theory Ser. A}, 116(1):232--241, 2009.

\bibitem{klazar2000}
Martin Klazar.
\newblock The {F}\"{u}redi-{H}ajnal conjecture implies the {S}tanley-{W}ilf
  conjecture.
\newblock In {\em Formal power series and algebraic combinatorics ({M}oscow,
  2000)}, pages 250--255. Springer, Berlin, 2000.

\bibitem{KM2007}
Martin Klazar and Adam Marcus.
\newblock Extensions of the linear bound in the {F}\"{u}redi-{H}ajnal
  conjecture.
\newblock {\em Adv. in Appl. Math.}, 38(2):258--266, 2007.

\bibitem{MT2003}
Adam Marcus and G\'{a}bor Tardos.
\newblock Excluded permutation matrices and the {S}tanley-{W}ilf conjecture.
\newblock {\em J. Combin. Theory Ser. A}, 107(1):153--160, 2004.

\bibitem{mitchell1992}
Joseph S.~B. Mitchell.
\newblock {$L_1$} shortest paths among polygonal obstacles in the plane.
\newblock {\em Algorithmica}, 8(1):55--88, 1992.

\bibitem{pach1991vertical}
J\'{a}nos Pach and Micha Sharir.
\newblock On vertical visibility in arrangements of segments and the queue size
  in the {B}entley-{O}ttmann line sweeping algorithm.
\newblock {\em SIAM J. Comput.}, 20(3):460--470, 1991.

\bibitem{PT2006}
J\'{a}nos Pach and G\'{a}bor Tardos.
\newblock Forbidden paths and cycles in ordered graphs and matrices.
\newblock {\em Israel J. Math.}, 155:359--380, 2006.

\bibitem{Pettie2011c}
Seth Pettie.
\newblock On the structure and composition of forbidden sequences, with
  geometric applications.
\newblock In {\em Computational geometry ({SCG}'11)}, pages 370--379. ACM, New
  York, 2011.

\bibitem{Tardos2005}
G\'{a}bor Tardos.
\newblock On 0-1 matrices and small excluded submatrices.
\newblock {\em J. Combin. Theory Ser. A}, 111(2):266--288, 2005.

\bibitem{Tsai2023}
Shen-Fu Tsai.
\newblock Saturation of multidimensional 0-1 matrices.
\newblock {\em Discrete Math. Lett.}, 11:91--95, 2023.

\end{thebibliography}

\end{document}